\documentclass[preprint,12pt]{elsarticle}

\usepackage[hang]{footmisc}
\usepackage{amsmath,rotating}
\usepackage{amssymb}
\usepackage{amsthm}
\usepackage{subfigure}
\usepackage{psfrag}
\usepackage[sort]{cite}
\usepackage{float}
\usepackage{epic}
\usepackage{color}
\usepackage{bbm}
\usepackage{bm}
\usepackage{pstricks}
\usepackage{pst-plot}
\usepackage{tikz}
\usepackage{natbib}

\newcommand {\bx} {\mathbf{x}}

\newcommand {\br} {\mathbf{r}}

\journal{Comput. Methods Appl. Math. Eng.}
\journal{}

\begin{document}

\begin{frontmatter}

\title{Stochastic Galerkin Framework with Locally Reduced Bases for Nonlinear Two-Phase Transport in Heterogeneous Formations}

\author[UNI]{Per Pettersson\corref{cor1}}
\ead{per.pettersson@uib.no}

\author[STAN]{Hamdi A. Tchelepi}
\ead{tchelepi@stanford.edu}

\cortext[cor1]{Corresponding Author: Per Pettersson}
\address[UNI]{Uni Research CIPR, PO Box 7800, N-5020 Bergen, Norway}
\address[STAN]{Department of Energy Resources Engineering, Stanford University, 367 Panama St., Stanford, CA 94305-2220, USA}

\begin{abstract}
\label{Abstract}
The generalized polynomial chaos method is applied to the Buckley-Leverett equation. We consider a spatially homogeneous domain modeled as a random field. The problem is projected onto stochastic basis functions which yields an extended system of partial differential equations. Analysis and numerical methods leading to reduced computational cost are presented for the extended system of equations. 

The accurate representation of the evolution of a discontinuous stochastic solution over time requires a large number of stochastic basis functions. Adaptivity of the stochastic basis to reduce computational cost is challenging in the stochastic Galerkin setting since the change of basis affects the system matrix itself. To achieve adaptivity without adding overhead by rewriting the entire system of equations for every grid cell, we devise a basis reduction method that distinguishes between locally significant and insignificant modes without changing the actual system matrices. 

Results are presented for problems in one and two spatial dimensions, with varying number of stochastic dimensions. We show how to obtain stochastic velocity fields from realistic permeability fields and demonstrate the performance of the stochastic Galerkin method with local basis reduction. The system of conservation laws is discretized with a finite volume method and we demonstrate numerical convergence to the reference solution obtained through Monte Carlo sampling.

\end{abstract}

\begin{keyword}
Polynomial chaos, Stochastic Galerkin method, Basis Reduction, Subsurface Flow
\end{keyword}

\end{frontmatter}

\section{Introduction}
Reliable numerical predictions of reservoir performance depend on the availability and quality of data describing the heterogeneous reservoir properties, e.g., permeability and porosity. Data are in general limited in quality and quantity, resulting in significant uncertainty in petrophysical reservoir properties. 
 Furthermore, the computational complexity of a realistic reservoir model imposes limits on the resolution of the numerical model, resulting in unresolved subscale phenomena and subsequent uncertainty in effective parameter properties. Uncertainty due to lack of data and complex physical correlation on multiple scales can often be be described within a stochastic framework to estimate risk or establish confidence in predictions used to make decisions in reservoir management. Propagation of input uncertainty through nonlinear models calls for accurate and robust methods to satisfactorily quantify the effects on quantities of interest. Forward uncertainty quantification aims at providing methods to accurately represent and propagate input uncertainty to the flow-performance predictions in a systematic way.

In the stochastic hydrology community, statistical moment equation methods have been popular in order to quantify uncertainty in aquifers with analytical, or semi-analytical, methods \citep{Gelhar_86, Zhang_02}. The variables and material parameters are decomposed into mean and fluctuation, or as an infinite series, and partial differential equations are derived for the moments of contaminant concentrations or hydraulic head, for example. Examples of moment equation methods include \citep{Neuman_Orr_93, Guadagnini_Neuman_99} where integro-differential formulations for the moments of the flux function based on measurements were derived and solved using a finite element method. The perturbation approach in combination with Karhunen-Lo\`{e}ve (KL) decomposition based on the covariance kernel of the conductivity was investigated in \citep{Zhang_Lu_04}. The closure approximation of the partial differential equations derived in these methods may be inaccurate \citep{Jarman_Tartakovsky_13}, and the resulting expressions are limited to small perturbations, i.e., small variance of input parameters \citep{Caroni_Fiorotto_05}. 

Stochastic analysis of nonlinear two-phase transport - the so-called Buckley-Leverett problem - based on moment equations in one and two spatial dimensions was presented in \citep{Zhang_Tchelepi_99} and \citep{Zhang_Li_Tchelepi_00}.
A cumulative density function (CDF) method for the Buckley-Leverett equations with random total volumetric flux was introduced in \citep{Wang_etal_13}. They derived a partial differential equation (PDE) for a Heaviside function involving the time dependent stochastic solution, and the associated CDF was obtained by integration with respect to the probability measure of the saturation.

The generalized polynomial chaos (gPC) framework introduced in the seminal work in \citep{Ghanem:1991} and extended in~\citep{XiuKarniadakis02} allows for representing uncertainty in PDE solutions as functions of uncertainty in the input parameters. 
Unlike the small perturbation assumption frequently made in subsurface flow \citep{Gelhar_86}, the gPC framework imposes no formal constraint on the variance of the solution. The possible limitations in the framework lie instead in the fact that small-scale variations may require large numbers of Karhunen-Lo\`{e}ve terms and lead to prohibitive computational cost. Provided that the velocity field can be accurately represented with a moderately small number of Karhunen-Lo\`{e}ve expansion terms, the KL-gPC framework should prove to be superior to perturbation based methods. Flow and transport in porous media have been treated within a polynomial chaos framework previously, but most efforts and reported success have been achieved for the flow problem, e.g., data-driven polynomial chaos for flow in CO$_2$ storage \citep{Oladyshkin_etal_2011} and infiltration problems in hydrology~\citep{Sochala_LeMaitre_13}. The log transformed hydraulic conductivity field was represented by a KL expansion and the hydraulic head obtained through resepectively Monte Carlo sampling, probabilistic collocation and the stochastic Galerkin methods in \citep{Li_Zhang_07}. A similar stochastic flow model was used in \citep{Muller_etal_11}, followed by model reduction of the flow field for efficient Monte Carlo simulation of the transport problem.

In this work, we extend the generalized polynomial chaos and stochastic Galerkin framework for the Buckley-Leverett equations and present analysis for stochastic fractional flow. The Buckley-Leverett solution is discontinuous, both in the physical and stochastic spaces. Therefore, we use a localized multiwavelet basis~\citep{LeMaitre_etal_04} to represent uncertainty instead of the classical orthogonal polynomials. A multi-element partitioning of the stochastic space \citep{Wan_Karniadakis_05} in combination with local wavelet bases on each stochastic element was applied to the Buckley-Leverett problem in \citep{Koppel_etal_14}. The stochastic multi-elements are mutually independent which allows parallell computation of the element-wise solutions, subsequently assembled to the global solution through a fast post-processing step. The stochastic multi-element method is computationally efficient and well suited to be combined with the methods presented in this paper. However, this will not be pursued here since the focus is somewhat different. Depending on implementation, we have observed lack of convergence to the true solution and further analysis of the problem is therefore necessary. Numerical properties related to convergence were investigated in a single-dimensional setting in~\citep{Pettersson_Tchelepi_14}. In this paper, the focus is on the properties of the stochastic Galerkin system, the representation and resolution of stochastic variables representing the uncertainty, and increasing the computational efficiency through stochastic basis reduction.

The Buckley-Leverett problem is described by a nonlinear hyperbolic PDE, and a stochastic formulation based on generalized polynomial chaos is expected to require a large number of stochastic basis functions for accurate representation. Nevertheless, we believe that it is an attractive alternative to existing methods for solving hyperbolic stochastic PDEs, especially when the wave nature of the problem is exploited to obtain local basis reduction. Sampling based generalized polynomial chaos methods, such as stochastic collocation \citep{Xiu_Hesthaven_05} suffer from the curse of dimensionality, and they become infeasible for large problems due to the prohibitive computational cost. With the continuous growth of computer power, stochastic Galerkin methods including the efficient adaptive and parallelized hybrid stochastic Galerkin solver in~\citep{Burger_etal_14}, have been gaining popularity as a powerful alternative to sampling based methods.

\subsection{Contributions of this Work}
In this paper, we provide analysis of the full stochastic Galerkin system with pre-computed integrals involving stochastic basis functions  and prove that the proposed formulation is hyperbolic. The full stochastic Galerkin projection results in an extended system of equations, but the benefit is that there is no need for repeated expensive numerical quadrature during the flux evaluations. This is quite different from the approach taken in previous work, e.g.~\citep{Tryoen_etal_10,Burger_etal_14}, and therefore their hyperbolicity analysis cannot be extended to this case.  

We outline a methodology to obtain saturation statistics from covariance functions of input parameters using the Karhunen-Lo\`{e}ve expansion and the gPC framework with multiwavelets. Stochastic velocity fields are obtained for a line-injection problem based on covariance functions for the velocity components derived in~\citep{Rubin_wrr_90}. Karhunen-Lo\`{e}ve expansions are also employed to represent a stochastic permeability field that is subsequently used as input to a numerical pressure solver for a quarter five-spot problem. The velocity field is obtained via Darcy's law, and represented as a gPC expansion used to compute the solution to the Buckley-Leverett problem.

We introduce a new locally reduced-order stochastic Galerkin method to alleviate the computational cost of high-order gPC representations of the uncertainties. The local reduction of the order of multiwavelet expansion does not adversely affect the accuracy of the solution, but reduces the numerical cost by neglecting multiwavelet modes that are locally insignificant. The performance of the method is demonstrated for one- and two-dimensional (in physical space) problems, and for both single and multiple stochastic dimensions. Depending on the problem setup, the speedup observed is up to more than an order of magnitude compared to the standard stochastic Galerkin implementation without local basis reduction or other means of adaptivity.


\subsection{Deterministic Formulation of the Buckley-Leverett Equation}
\label{sec:det_bl}
Let $S$ denote the water saturation in a subsurface porous medium, modeled by the Buckley-Leverett equation 
\begin{equation}
\label{eqn:bl_conv_gen}
\phi(x,y) \frac{\partial S}{ \partial t} + \nabla \cdot \left( \bm{q} f(S) \right)  = 0, \quad x,y \in D, t \geq 0,
\end{equation}
defined on a spatial domain $D$, with the scalar flux function $f$ and porosity $\phi$. The total flux $\bm{q} = (q_{(x)}, q_{(y)})$ is given by Darcy's law:
\[
\bm{q} = -\lambda_{T}(S)\nabla p,
\]
where $p$ is pressure and $\lambda_{T}$ is the total mobility, defined as
\[
\lambda_T = k\left( \frac{k_{r_w}}{\mu_{w}} + \frac{k_{r_n}}{\mu_{n}} \right).
\]
Here, $k$ denotes absolute permeability, $k_{r_\alpha}$ ($\alpha=w,n$) is the relative permeability of the wetting ($w$) and non-wetting ($n$) phases, respectively, and $\mu_{\alpha}$ ($\alpha=w,n$) denotes the phase viscosity. We invoke the incompressibility condition
\begin{equation}
\label{eq:inc_cond}
\nabla \cdot \bm{q} = 0.
\end{equation}
To provide an input flux $\bm{q}$ to \eqref{eqn:bl_conv_gen}, we solve \eqref{eq:inc_cond} with a measured, or estimated, total mobility (conductivity) field as input. In the multiphase case, the total mobility takes the role of hydraulic conductivity. Due to limited site-specific measurements, stochastic models for the input are commonly assumed. The permeability field is typically unknown due to the infeasibility of including an exact geologic description. This results in uncertainty in the total mobility.

We rewrite (\ref{eqn:bl_conv_gen}) in the form
\begin{equation}
\label{eqn:bl_conv}
\phi(x,y) \frac{\partial S}{ \partial t} + q_{(x)}\frac{\partial f(S)}{\partial x} + q_{(y)}\frac{\partial f(S)}{\partial y}=0, \quad x, y\in D, t\geq 0.
\end{equation}

\noindent
With the relative-permeability functions $k_{r_w}=S^2$ and $k_{r_n}=(1-S)^2$, the fractional flow flux function $f$ is given by
\begin{equation}
\label{eq:buck_lev_flux}
f(S)=\frac{\frac{k_{r_w}}{\mu_w}}{ \frac{k_{r_w}}{\mu_w} + \frac{k_{r_n}}{\mu_n}}=\frac{S^2}{S^2+a(1-S)^2},
\end{equation}
where $a=\mu_{w}/\mu_{n}$. For analysis of hyperbolicity of the stochastic formulation, we will find it useful to express the Buckley-Leverett equation (\ref{eqn:bl_conv}) in quasi-linear form,
\begin{equation}
\label{eqn:bl_nonconv}
\phi(x,y)\frac{\partial S}{\partial t} + q_{(x)}f'(S) \frac{\partial S}{\partial x} + q_{(y)}f'(S) \frac{\partial S}{\partial y} = 0,
\end{equation}
where
\[
f'(S) = \frac{2aS(1-S)}{(S^2+a(1-S)^2)^2}.
\]

%
%

\section{Representation of Uncertainty}
\label{sec:rep_of_unc}
When the covariance function of a random field is known, it is optimally represented in a spectral expansion by the \textit{Karhunen-Lo\`{e}ve} expansion \citep{ka46, Loeve_77}. The covariance function is often quite difficult to obtain, even for static input parameters (e.g., permeability), and we are interested in quantifying the uncertainty in output parameters that result from a nonlinear dynamic process (multiphase transport). For output quantities of interest, the covariance function is in general not known at all, but we can combine the Karhunen-Lo\`{e}ve framework with multiwavelet expansions to represent input random fields, e.g., permeabilities. The multiwavelet representation is then employed in order to propagate the uncertainty to the outputs of interest.

\subsection{Karhunen-Lo\`{e}ve Expansion}
\label{sec:kl_exp}
Consider a vector valued random field $\bm{g}(x,y)$ in two spatial dimensions with known matrix-valued covariance function $\mathbf{C}_{\bm{g}}(\mathbf{x},\mathbf{x}')$ where $\bx=(x,y)$. In the numerical experiments, $\bm{g}$ will either be equal to the velocity field $\bm{u}=\bm{q}/\phi$, or the total mobility determined by the log-permeability field $\log(k)$ which may be represented as a Gaussian random field~\citep{Hoeksema_Kitanidis_85}. A realistic covariance model for the velocity field must in general be based on numerical simulation of the pressure distributions. Following \citep{Perrin_etal_13}, the field $\bm{g}$ with expectation $\overline{\bm{g}}$ can be approximated by the truncated Karhunen-Lo\`{e}ve expansion
\begin{equation}
\label{eq:kl_exp}
\bm{g}(x,y,\boldsymbol{\xi})= \overline{\bm{g}}(x,y)  +\sum_{k=1}^{d}\sqrt{\lambda_{k}} \bm{g}_{k}^{\scriptscriptstyle{\textup{KL}}}(x,y)\xi_{k},
\end{equation}
where $\boldsymbol{\xi} = (\xi_1,...,\xi_d)$ is a random vector where the entries are uncorrelated random variables with probability measure $\mathcal{P}$. The generalized eigenpairs $(\lambda_k, \bm{g}_k^{ \scriptscriptstyle{\textup{KL}} })$ can be determined from the solution of the generalized eigenvalue problem
\begin{equation}
\label{eq:gen_evpr_kl}
\int_{D}\mathbf{C}_{\bm{g}}(\bx,\bx') \bm{g}_{k}^{\scriptscriptstyle{\textup{KL}}}(\bx')d\bx' = \lambda_{k} \bm{g}_{k}^{\scriptscriptstyle{\textup{KL}}}(\bx), \quad k \in \mathbb{N}^{+}.
\end{equation}
The accuracy of the truncated KL expansion is determined by the order of expansion $d$ and the decay of the eigenvalues $\{ \lambda_k \}$. The Karhunen-Lo\`{e}ve expansion is bi-orthogonal, i.e.  
\begin{eqnarray}
&&\left\langle  \bm{g}_{j}^{\scriptscriptstyle{\textup{KL}}}(\bx), \bm{g}_{k}^{\scriptscriptstyle{\textup{KL}}}(\bx) \right\rangle_{D} \equiv \int_{D}\left( \bm{g}_{j}^{\scriptscriptstyle{\textup{KL}}}(\bx) \right)^{\scriptscriptstyle{T}} \bm{g}_{k}^{\scriptscriptstyle{\textup{KL}}}(\bx) d\bx = \delta_{jk}, \\
&&\left\langle \xi_{j} \xi_{k} \right\rangle_{\Omega} \equiv \int_{\Omega} \xi_{j} \xi_{k} d\mathcal{P}=\delta_{jk}.
\end{eqnarray}
The probability distribution of the random variable $\xi_{k}$ is defined by
\begin{equation}
\label{eq:eta_defs}
\xi_k =  \frac{1}{\sqrt{\lambda_{k}}} \left\langle  \bm{g} - \overline{\bm{g}} , \bm{g}_{k}^{\scriptscriptstyle{\textup{KL}}} \right\rangle_{D}, \quad k \in \mathbb{N}^{+}. 
\end{equation}
In order to exploit the richness of the stochastic representation by accurately estimating $\xi_k$, excessive sampling of~\eqref{eq:eta_defs} would be required. In the following however, we will \textit{assume} that we have sufficient information and choose some probability law for $\{\xi_{k}\}$. A common assumption on the probability law of $\{ \xi_k \}$ is Gaussian distribution. In this case the fact that the random variables are uncorrelated also implies that they are independent. Note that in general, the entries of the vector $\bm{g}$ are not independent. This will be the case for the random velocity field in \citep{Rubin_wrr_90} used in Section \ref{sec:num_res_rub90}.


It is common in subsurface flow to assume the exponential covariance function, attractive for its analytical expressions for eigenvector-eigenvalue pairs. For the log-permeability field $Y=\log(k)$ on a domain of size $L_x \times L_y$, we assume the separable exponential covariance function
\[
C_{Y}(\bx,\bx') = \sigma_{Y}^{2} \exp \left(-\frac{|x-x'|}{l_{x}} -\frac{|y-y'|}{l_{y}} \right),
\]
with correlation lengths $l_{x}$, $l_{y}$, variance $\sigma_{Y}^2$ and $\bx = (x,y)$. For this covariance function, the eigenvalues $\lambda_{k}$ and eigenfunctions $\bm{g}_{k}^{\scriptscriptstyle{\textup{KL}}}$ in~\eqref{eq:kl_exp} are the ordered products of the one-dimensional eigenpairs, given by respectively
\[
\lambda_{k}=\frac{4 l_{x} l_{y}\sigma_{Y}^{2}}{ \left(1+(\omega_{k_x} l_{x})^2 \right) \left(1+(\omega_{k_y} l_{y})^2 \right)} ,
\]
and
\[
\bm{g}_{k}^{\scriptscriptstyle{\textup{KL}}}(\bx)=
\frac{ \left( l_{x}\omega_{k_{x}}\cos(\omega_{k_{x}}x)+\sin(\omega_{k_{x}}x) \right) \left( l_{y}\omega_{k_{y}}\cos(\omega_{k_{y}}y)+\sin(\omega_{k_{y}}y) \right) }{\sqrt{ \left((l_{x}^2\omega_{k_{x}}^{2}+1)L_{x}/2+l_{x} \right) \left((l_{y}^2\omega_{k_{y}}^{2}+1)L_{y}/2+l_{y} \right) }},
\]
where $\omega_{k_{\alpha}}$ are the ordered ($\omega_{k_{\alpha}}<\omega_{k_{\alpha}+1}$) positive solutions of 
\[
(l_{\alpha}^2\omega^2-1)\sin(\omega L_{\alpha}) = 2l_{\alpha}\omega \cos(\omega L_{\alpha}), \quad \alpha = x,y.
\]

%
%

\subsection{Generalized Polynomial Chaos Formulation}
Let $\boldsymbol{\xi} = (\xi_1,\dots,\xi_d)^{T}$ be a vector of random variables parameterizing the uncertainties of the input fields defined on a space $(\Omega, \mathcal{F}, \mathcal{P})$, where $\Omega$ is the set of events, $\mathcal{F}$ is the $\sigma$-algebra and $\mathcal{P}$ is the probability measure. We will assume that the entries of $\boldsymbol{\xi}$ are independent and identically distributed. In this work, $\boldsymbol{\xi}$ is identical to the random vector of the KL expansion introduced in Section~\ref{sec:kl_exp}. The inner product of two functionals $v(\boldsymbol{\xi}),w(\boldsymbol{\xi}) \in L^2(\Omega, \mathcal{P})$ is defined by
\[
\langle v(\boldsymbol{\xi}), w(\boldsymbol{\xi}) \rangle = \int_{\Omega} v(\boldsymbol{\xi}) w(\boldsymbol{\xi}) d\mathcal{P}(\boldsymbol{\xi}).
\]
The inner product induces a norm denoted $\left\| . \right\|_{L^2(\Omega,\mathcal{P})}$.

For $l=1,\dots,d$, let $\{ \psi_{k}(\xi_{l}) \}_{k=0}^{\infty}$ be a univariate basis of e.g., polynomials that are orthonormal w.r.t. the measure of $\xi_{l}$, i.e.
\[
\langle \psi_{i}(\xi_{l}), \psi_{j}(\xi_{l}) \rangle = \delta_{ij}, \mbox{ for all }l=1,\dots, d.
\]
We obtain a multi-variate basis $\{ \psi_{\footnotesize \bm{k}} \}$ by the tensor product
\[
\psi_{\footnotesize \bm{k}}(\boldsymbol{\xi}) = \prod_{l=1}^{d}\psi_{k_l}(\xi_l)
\]
with the multi-index $\bm{k} \in \mathbb{N}_{0}^{d}:=\left\{(k_1,\cdots,k_d): k_l\in \mathbb{N}\cup\{0\} \right\}$. For practical calculations, the multi-index $\bm{k}$ must be truncated in order to generate a finite cardinality basis. We use a total order $p$ basis defined by the index set 
\begin{equation}
\label{eq:complete}
\Lambda_{p,d} = \left\{\bm{k}\in\mathbb{N}_{0}^{d}: \Vert \bm{k}\Vert_{1}\leq p\right\}.
\end{equation} 
The total number of basis functions is then
\[
P=\frac{(p+d)!}{p!d!}.
\] 
To facilitate notation, we re-index the stochastic basis functions starting from 1. Any $f \in L^2(\Omega,\mathcal{P})$ can be represented by the generalized polynomial chaos expansion
\begin{equation}
\label{eq:def_gpc}
f = \sum_{\bm{k} \in \mathbb{N}_{0}^{d} }f_{\bm{k}}\psi_{\bm{k}}(\boldsymbol{\xi}) \approx \sum_{\bm{k} \in \Lambda_{p,d} }f_{\bm{k}}\psi_{\bm{k}}(\boldsymbol{\xi}) = \sum_{k=1}^{P}f_{k}\psi_{k}(\boldsymbol{\xi}),
\end{equation}
where the coefficients $f_{\bm{k}}$ are defined by the projections
\begin{equation}
\label{eq:proj_coe}
f_{\bm{k}}=\left\langle f, \psi_{\bm{k}}(\boldsymbol{\xi})\right\rangle, \quad \bm{k} \in \mathbb{N}^{d}.
\end{equation}
Then, in single-index notation
\[
\left\| f - \sum_{k=1}^{P}f_{k}\psi_{k}(\boldsymbol{\xi}) \right\|_{L^2(\Omega,\mathcal{P})} \rightarrow 0 \mbox{ as } P \rightarrow \infty.
\]

Approximation of the integrals to evaluate the gPC coefficients (\ref{eq:proj_coe}) in general involves quadrature in multiple dimensions. Tensorization of univariate quadrature rules, e.g. Gauss quadrature,  quickly become infeasible as the number of stochastic dimensions grows. Sparse quadrature rules offer a computationally attractive alternative. Smolyak's rule based on a subset of weights and points from a range of tensorized quadrature rules \citep{Smolyak_63} is an important example. Instead of computing the stochastic integrals in every temporal and spatial grid point, we derive an extended system based on stochastic projection. Thus, all stochastic integrals are precomputed, and we are left with solving an extended initial-boundary-value problem.

Some input and output quantities of interest (porosity, saturation) are restricted to intervals that represent physically meaningful solutions. To ensure that stochastic variation does not lead to unbounded quantities, we use stochastic basis functions with bounded support to represent uncertainty.

\subsection{Multiwavelet Representation of Stochastic Functions}
Wavelet bases obtained from multi-resolution analysis for representation of stochastic functions were introduced in~\citep{LeMaitre_etal_04}. Wavelets are localized basis functions organized in a hierarchy of different resolution levels. Each successive resolution level represents finer features of the solution. The wavelets have non-overlapping support within each resolution level, and in this sense they are localized. Still, the basis is global due to the overlapping support of wavelets belonging to different resolution levels. The localization of the basis alleviates the Gibbs phenomenon in the proximity of discontinuities in the stochastic dimension.

We start by briefly describing the multiwavelet basis in a single dimension. A more detailed account can be found in~\citep{LeMaitre_etal_04}. Let $\mathbf{V}_{N_p}$ be the space of polynomials of degree at most $N_p$ defined on the interval $[-1,1]$. The construction of multiwavelets aims at finding a basis of piecewise polynomials for the orthogonal complement of $\mathbf{V}_{N_p}$ in the space $\mathbf{V}_{N_{p}+1}$ of polynomials of degree at most $N_{p}+1$. Merging the bases of $\mathbf{V}_{N_p}$ and that of the orthogonal complement of $\mathbf{V}_{N_p}$ in $\mathbf{V}_{N_{p}+1}$, we obtain a piecewise polynomial basis for $\mathbf{V}_{N_{p}+1}$. Continuing the process of finding orthogonal complements in spaces of increasing degree of piecewise polynomials, leads to a basis for $L_2([-1,1])$.

We first introduce a smooth polynomial basis on $[-1,1]$. Let $\{Le_i(\xi) \}_{i=0}^{\infty}$ be the set of Legendre polynomials that are defined on $[-1,1]$ and orthogonal with respect to the uniform measure. The set $\{Le_{i}(\xi) \}_{i=0}^{N_p}$ is an orthonormal basis for $\mathbf{V}_{N_p}$.

 Following the algorithm in \citep{Alpert_93}, we next introduce a set of mother wavelets $\left\{ \psi_{i}^{W}(\xi) \right\}_{i=0}^{N_p}$ defined on the domain $\xi \in [-1,1]$. 
 By construction, the members of the set $\{ \psi_{i}^{W}(\xi) \}_{i=0}^{N_p}$ are orthogonal to all polynomials of order at most $N_p$, hence the wavelets are orthogonal to the set $\{ Le_{i}(\xi) \}_{i=0}^{N_p}$ of Legendre polynomials. Based on translations and dilations of the mother wavelets, we get the wavelet family
\[
\psi_{i,j,k}^W(\xi)=2^{j/2}\psi_i^W(2^j \xi-k), \qquad i=0,...,N_p, \quad j=0,1,..., \quad k=0,...,2^{j-1}.
\]
We can now define the multiwavelet basis $\{ \psi_{m} \}_{m=0}^{\infty}$ as follows. Let $\psi_m(\xi)$ for $m=0,...,N_p$ be the set of Legendre polynomials up to order $N_p$, and concatenate the indices $i,j,k$ into $m=(N_p+1)(2^j+k-1)+i$ so that $\psi_m(\xi)\equiv \psi_{i,j,k}^{W}(\xi)$ for $m > N_p$. With the multiwavelet basis $\{ \psi_m(\xi)\}_{m=0}^{\infty}$ we can represent any random variable $u(x,t,\xi)$ with finite variance as
\[
u(x,t,\xi)=\sum_{m=0}^{\infty}u_{m}(x,t)\psi_{m}(\xi),
\]
which is in the form (\ref{eq:def_gpc}) with the global polynomials replaced by piecewise polynomials. In the computations, we truncate the wavelet series both in terms of the piecewise polynomial order $N_p$ and the resolution level $N_r$. We retain $P=(N_p+1)2^{N_r}$ terms of the multiwavelet expansion. In multiple dimensions, we use a total order basis within each resolution level, i.e. a total of $P=(p+d)!2^{N_r}/(p!d!)$ basis functions.

The truncated wavelet basis is characterized by the piecewise polynomial order $N_p$ and the number of resolution levels $N_r$. As special cases of the wavelet basis, we obtain the Legendre polynomial basis for $N_r=0$ ($i=j=0$), and the Haar wavelet basis of piecewise constant functions for $N_p=0$. Since the Legendre polynomials are a subset of the multiwavelets, we will henceforth use the term multiwavelets (MW) to denote all the sets of basis functions in this paper.

%
%

\section{Stochastic Galerkin Formulation} 
Inserting the MW expansions of $S$ and $\bm{u} = \bm{q}/\phi$ into (\ref{eqn:bl_conv}) and truncating to a finite order $P$, we arrive at
\begin{equation}
\label{eq:bl_w_gpc_v}
 \sum_{j=1}^{P}\frac{\partial S_j(\mathbf{x},t) }{\partial t} \psi_{j}(\boldsymbol{\xi}) + \left( \sum_{i=1}^{P}\bm{u}_{i} \psi_{i}(\boldsymbol{\xi}) \right) \cdot   \nabla   f \left( \sum_{j=1}^{P}S_j(\mathbf{x},t) \psi_{j}(\boldsymbol{\xi}) \right)  = 0
\end{equation}
Multiplying (\ref{eq:bl_w_gpc_v}) by $\psi_{k}(\boldsymbol{\xi})$, rearranging and integrating w.r.t. the measure $\mathcal{P}$ and using the orthogonality of the stochastic basis, we get
\begin{equation}
\label{eq:bl_sg_mat}
 \frac{\partial S_{k}(\mathbf{x},t) }{\partial t} + \left\langle \left( \sum_{i=1}^{P}\bm{u}_{i} \psi_{i} \right) \cdot \nabla f\left(\sum_{j=1}^{P}S_j \psi_j \right) \psi_{k} \right\rangle = 0, \quad \mbox{ for }k=1,\dots, P,
\end{equation}
where $\left\langle . \right\rangle$ denotes the expectation operator w.r.t. $\boldsymbol{\xi}$.
For completeness of the presentation and for comparison with the analysis later in this paper, we next reproduce the analysis of hyperbolicity for the Buckley-Leverett problem that was previously presented for related but different problems in~\citep{Tryoen_etal_10,Burger_etal_14}.
\newtheorem{thm}{Proposition}  
  \begin{thm}
  \label{thm:class_hyp_pr}
The stochastic Galerkin formulation (\ref{eq:bl_sg_mat}) of any order $P \in \mathbb{N}_{0}^{+}$ is hyperbolic.
  \end{thm}
  \begin{proof}
The semi-linear form of (\ref{eq:bl_sg_mat}) in matrix form can be written
\[
\frac{\partial \bm{S}^{P}}{\partial t} + \bm{J}_{x}(\bm{S}^{P})\frac{\partial \bm{S}^{P}}{\partial x} + \bm{J}_{y}(\bm{S}^{P}) \frac{\partial \bm{S}^{P}}{\partial y} = 0,
\]
where $\bm{S}^{P}=(S_1,\dots,S_{P})^T$, and the Jacobian matrices are
\begin{align}
\label{eq:Jx_def}
& [\bm{J}_{x}(\bm{S}^{P})]_{kj} = \left\langle \left( \sum_{i=1}^{P}v^{(x)}_{i} \psi_i \right) f'\left(\sum_{j'=1}^{P}S_{j'} \psi_{j'} \right) \psi_{j}\psi_{k} \right\rangle, \\
\label{eq:Jy_def}
& [\bm{J}_{y}(\bm{S}^{P})]_{kj} = \left\langle \left( \sum_{i=1}^{P}v^{(y)}_{i} \psi_i \right) f'\left(\sum_{j'=1}^{P}S_{j'} \psi_{j'} \right) \psi_{j}\psi_{k} \right\rangle.
\end{align}
It follows from (\ref{eq:Jx_def}) and (\ref{eq:Jy_def}) that $\bm{J}_{x}$ and $\bm{J}_{y}$ are both symmetric. Thus, for any real constants $c_1,c_2$, the matrix $c_1 \bm{J}_{x}+c_2 \bm{J}_{y}$ is also symmetric and has an eigenvector decomposition with real eigenvalues. It follows that (\ref{eq:bl_sg_mat}) is hyperbolic for any $P\in \mathbb{N}_{0}^{+}$.
\end{proof}

%
%

\subsection{Pseudo-spectral Flux Approximation}
A numerical method for the stochastic Galerkin Buckley-Leverett problem requires the approximation of stochastic integrals over nonlinear functions of $\bm{S}^{P}$  (i.e., the matrices $\bm{J}_{x}(\bm{S}^{P})$ and $\bm{J}_{y}(\bm{S}^{P})$ defined in (\ref{eq:Jx_def}) and (\ref{eq:Jy_def})) at each spatial point and at each time step. 
Likewise, in order to evaluate the stochastic Galerkin flux function numerically, we also need to compute 
\begin{equation}
\label{eq:inpr_flux}
\bm{f}^{P}_{k} = \left\langle f(S) \psi_k\right\rangle,
\end{equation}
 where
\[
f(S) \approx \frac{ \left(\sum_{i=1}^{P}S_i\psi_i \right) \left(\sum_{k=1}^{P}S_k\psi_k \right)}{\left(\sum_{l=1}^{P}S_l\psi_l \right) \left(\sum_{m=1}^{P}S_m\psi_m \right) + a \left(1-\sum_{l=1}^{P}S_l\psi_l \right) \left(1- \sum_{m=1}^{P}S_m\psi_m  \right)}.
\]
In the remainder of this Section, we will only treat the $x$-dimension in the analysis, since the results for the $y$-dimension are analogous. Repeated calculations of (\ref{eq:inpr_flux}) are costly due to the stochastic integrals over rational expressions of the basis functions. Even though direct evaluation of these integrals through, say, a low-order quadrature rule may indeed be feasible to perform online, we would ideally like to avoid online numerical quadrature for the stochastic Galerkin system. To alleviate the computational cost, we introduce an approximation based on repeated application of the \textit{pseudo-spectral approximation} \citep{Debusschere}. The idea behind the pseudo-spectral approximation is to replace complex and costly stochastic integrals with a series of less complex and less costly stochastic integrals that can be pre-computed and used throughout the simulation. It will be practical to express algebraic operations on MW representations of stochastic variables as sequences of matrix-vector multiplications. For any functions $a(\boldsymbol{\xi}), b(\boldsymbol{\xi})$ of order $P$ of MW approximations in vector form, $\bm{a}^{P}, \bm{b}^{P} \in \mathbb{R}^{P}$, we define the matrices $\bm{A}(.)$ and $\bm{B}(.,.)$,
\begin{align}
\label{eq:a_def}
 [\bm{A}(\bm{a}^{P})]_{jk} &\equiv \sum_{i=1}^{P} \left\langle \psi_i \psi_j \psi_k \right\rangle
a_i, \quad j,k=1,\dots,P. \\
\label{eq:b_def}
 [\bm{B}(\bm{a}^{P}, \bm{b}^{P})]_{jk} &\equiv \sum_{h=1}^{P} \sum_{i=1}^{P}  \langle \psi_{h} \psi_{i} \psi_{j} \psi_{k}\rangle a_{h} b_{i}, \quad j,k=1,\dots,P. 
\end{align}
for inner triple and inner quadruple products of basis functions, respectively. For the computation of the MW coefficients of the product  $p(\xi)$ of two stochastic functions $a(\xi)$ and $b(\xi)$, we have
\[
p_{k} = \sum_{i=1}^{P} \sum_{j=1}^{P} \langle \psi_{i} \psi_{j} \psi_{k} \rangle a_{i} b_{j},
\]
or, in matrix-vector notation,
\[
\bm{p}^{P} = \bm{A}(\bm{a}^{P})\bm{b}^P.
\]
To compute the MW coefficients of the stochastic product of three functions,
\[
d(\xi)=a(\xi)b(\xi) c(\xi),
\]
we get, by projection onto the first $P$ basis functions,
\[
d_{k} = \sum_{h=1}^{P} \sum_{i=1}^{P} \sum_{j=1}^{P} \langle \psi_{h} \psi_{i} \psi_{j} \psi_{k} \rangle a_{h} b_{i} c_{j}, \quad k=1,...,P. 
\]
Equivalently, this product can be written in matrix-vector notation, where the vector of MW coefficients $\bm{d}^{P} = (d_{1},\dots,d_{P})^{T}$ are given by,
\[
\bm{d}^{P} = \bm{B}(\bm{a}^{P}, \bm{b}^{P}) \bm{c}^{P}.
\]

Using the matrix notation introduced above, the stochastic Galerkin flux function for the Buckley-Leverett problem is defined by the vector of MW coefficients $\bm{f}^{\scriptscriptstyle{P}}$ that satisfies the linear system
\begin{equation}
\label{eq:quad_sg_flux}
\left[ \bm{B}(\bm{S}^{P},\bm{S}^{P})+a \bm{B}(\bm{e}_{1}-\bm{S}^{\scriptscriptstyle{P}}, \bm{e}_{1} - \bm{S}^{\scriptscriptstyle{P}}) \right] \bm{f}^{\scriptscriptstyle{P}} = \bm{B}(\bm{S}^{P},\bm{S}^{P})\bm{u}^{\scriptscriptstyle{P}}.
\end{equation}
Note that the proof of hyperbolicity of Proposition~\ref{thm:class_hyp_pr} does not hold for the stochastic Galerkin flux \eqref{eq:quad_sg_flux} based on evaluation of precomputed stochastic inner products. Next, we will present proof of hyperbolicity applicable to this problem flux formulation.
\begin{thm}
\label{thm:hyperb_pssp}
Let $\bm{B}(\bm{S}^{P},\bm{S}^{P})+a \bm{B}(\bm{e}_{1}-\bm{S}^{\scriptscriptstyle{P}}, \bm{e}_{1} - \bm{S}^{\scriptscriptstyle{P}})$ be positive definite and let $P$ be any order of MW approximation. Then the stochastic Galerkin formulation with the flux (\ref{eq:quad_sg_flux}) is hyperbolic. 
\end{thm}
\begin{proof}
By the chain rule, for $k=1,\dots, P$,
\begin{multline}
\frac{\partial \bm{f}^{\scriptscriptstyle{P}}}{\partial S_{k} }= -2
\left[ \bm{B}(\bm{S}^{P},\bm{S}^{P})+a \bm{B}(\bm{e}_{1}-\bm{S}^{\scriptscriptstyle{P}}, \bm{e}_{1} - \bm{S}^{\scriptscriptstyle{P}}) \right]^{-1}  \\
\times
\left[
\bm{B}(\bm{S}^{P},\bm{e}_{k}) - a \bm{B}(\bm{e}_{1}-\bm{S}^{\scriptscriptstyle{P}}, \bm{e}_{k} )
\right] 
 \left[ \bm{B}(\bm{S}^{P},\bm{S}^{P})+a \bm{B}(\bm{e}_{1}-\bm{S}^{\scriptscriptstyle{P}}, \bm{e}_{1} - \bm{S}^{\scriptscriptstyle{P}}) \right]^{-1}  \\
 \times \bm{B}(\bm{S}^{P},\bm{S}^{P})\bm{u}^{\scriptscriptstyle{P}} 
+ 2 \left[ \bm{B}(\bm{S}^{P},\bm{S}^{P})+a \bm{B}(\bm{e}_{1}-\bm{S}^{\scriptscriptstyle{P}}, \bm{e}_{1} - \bm{S}^{\scriptscriptstyle{P}}) \right]^{-1} \bm{B}(\bm{S}^{P},\bm{u}^{\scriptscriptstyle{P}}) \bm{e}_{k}.
\end{multline}
Let $\bm{w}^{\scriptscriptstyle{P}} = \left[ \bm{B}(\bm{S}^{P},\bm{S}^{P})+a \bm{B}(\bm{e}_{1}-\bm{S}^{\scriptscriptstyle{P}}, \bm{e}_{1} - \bm{S}^{\scriptscriptstyle{P}}) \right]^{-1} \bm{B}(\bm{S}^{P},\bm{S}^{P})$. Then
\begin{multline}
\frac{\partial \bm{f}^{\scriptscriptstyle{P}}}{\partial \bm{S}^{\scriptscriptstyle{P}} } = 2\left[ \bm{B}(\bm{S}^{P},\bm{S}^{P})+a \bm{B}(\bm{e}_{1}-\bm{S}^{\scriptscriptstyle{P}}, \bm{e}_{1} - \bm{S}^{\scriptscriptstyle{P}}) \right]^{-1} \\
\times  \left[ \bm{B}(a(\bm{e}_{1}-\bm{S}^{P})-\bm{S}^{P},\bm{w}^{P})+ \bm{B}(\bm{S}^{\scriptscriptstyle{P}}, \bm{u}^{\scriptscriptstyle{P}}) \right].
\end{multline}
Thus, the flux Jacobian $\partial \bm{f}^{P}/\partial \bm{S}^{P}$ is a product of two symmetric matrices,
\begin{align}
& \bm{M}_{1}=2\left[ \bm{B}(\bm{S}^{P},\bm{S}^{P})+a \bm{B}(\bm{e}_{1}-\bm{S}^{\scriptscriptstyle{P}}, \bm{e}_{1} - \bm{S}^{\scriptscriptstyle{P}}) \right]^{-1},  \nonumber \\ 
& \bm{M}_{2}=\left[ \bm{B}(a(\bm{e}_{1}-\bm{S}^{P})-\bm{S}^{P},\bm{w}^{P})+ \bm{B}(\bm{S}^{\scriptscriptstyle{P}}, \bm{u}^{\scriptscriptstyle{P}}) \right]. \nonumber
\end{align}
By assumption, $\bm{M}_{1}$ is positive definite (since its inverse is positive definite), and there exists an invertible and symmetric square root matrix, denoted $\bm{G}^{1/2}$, that satisfies $\bm{G}^{1/2} \bm{G}^{1/2} = \bm{M}_{1}$.  
Now, $\partial \bm{f}^{P}/\partial \bm{S}^{P} = \bm{G}^{1/2} \bm{G}^{1/2} \bm{M}_{2}$ is similar to the symmetric matrix $\bm{G}^{-1/2} \bm{G}^{1/2} \bm{G}^{1/2} \bm{M}_{2} \bm{G}^{1/2}= \bm{G}^{1/2} \bm{M}_{2} \bm{G}^{1/2}$. Since $\partial \bm{f}^{P}/\partial \bm{S}^{P}$ is similar to a diagonalizable matrix with real eigenvalues, $\partial \bm{f}^{P}/\partial \bm{S}^{P}$ is also diagonalizable with real eigenvalues, hence the flux formulation \eqref{eq:quad_sg_flux} is hyperbolic.
\end{proof}

The evaluation of the matrix $\bm{B}$ through fourth-order tensors is relatively costly, especially in a setting where the number of stochastic dimensions is large. Alternatively, one may instead compute products of three functions by successively computing the MW coefficients of products of two stochastic functions, i.e. introducing the approximation
\[
\bm{B}(\bm{a}^{P}, \bm{b}^{P})\bm{c}^{P} \approx \bm{A}(\bm{p}^{P})\bm{c}^P \mbox{ where }\bm{p}^{P} = \bm{A}(\bm{a}^{P})\bm{b}^P.
\]
The successive application of the matrix $\bm{A}(.)$ instead of $\bm{B}(.,.)$ is computationally efficient but introduces a stochastic aliasing error.

As an alternative to \eqref{eq:quad_sg_flux}, we may compute the vector of flux MW coefficients $\bm{f}^{P}$ by solving the linear system
\begin{equation}
\label{eq:trip_sg_flux}
\left[  \bm{A}(\bm{S}^{P})\bm{S}^{P} + a \bm{A}(\bm{e}_{1}- \bm{S}^{P})  (\bm{e}_{1}- \bm{S}^{P})  \right]  \bm{f}^{P}= \bm{A}(\bm{S}^{P}) \bm{S}^{P}  
\end{equation}
The approximation (\ref{eq:trip_sg_flux}) is based on successive application of pairwise pseudo-spectral products in order to avoid computation of stochastic inner products of higher order. The flux function (\ref{eq:trip_sg_flux}) is a suitable approximation if the error is negligible in comparison with (\ref{eq:inpr_flux}). 

%
%

\section{Numerical Methods}
\label{sec:num_meth}

%
%

\subsection{Reduced-order Stochastic Galerkin Method}
\label{sec:adapt_red_order}
Due to discontinuities and other sharp features, an accurate representation of a hyperbolic PDE solution in general requires high-order MW expansion. For solutions of ODEs, at a given time certain parts of the stochastic domain more significantly affects the solution than other regions of the domain. A temporally evolving adaptive stochastic basis might be an efficient alternative for these problems~\citep{LeMaitre_etal_04}. For a fixed point in physical space, we can expect similar behavior of  a hyperbolic PDE solution: certain regions of stochastic space need a finer basis for accurate representation than others. However, for a different spatial location, a different region of stochastic space will significantly impact the solution features. We can localize the stochastic basis in space and time, but we must maintain coupling since the solution at any spatial point will be affected by any given MW mode at some point in time. This makes adaptive methods for MW solutions of PDEs difficult. In particular, for stochastic Galerkin methods it is non-trivial to adaptively modify the stochastic basis by modifying the large system matrices involved. As an example, if one wants to compute the flux between two spatially adjacent grid points that are represented by different basis functions, one needs to complement both bases with the missing functions from the adjacent grid points to be able to propagate the flux. After that, one needs to perform basis reduction to maintain adaptivity of the bases, otherwise nothing is gained in computational cost by localizing the basis functions.

To avoid the cumbersome procedure described above, we devise an adaptive stochastic Galerkin method in the following way. The set of all admissible stochastic basis functions is determined and the inner products corresponding to the matrices $\bm{A}$ and $\bm{B}$ in~\eqref{eq:a_def} and~\eqref{eq:b_def} are computed and stored. For any given point $(x_{j},t_{n})$ in space and time, the computed solution vector $\bm{S}^{P}$ may have negligible entries that can be ignored in the flux evaluation. More generally, assuming that we want to perform a pseudospectral multiplication involving three MW vectors $\bm{a}^{P}$, $\bm{b}^{P}$, and $\bm{c}^{P}$, we first determine their local representations by only retaining the MW coefficients that are significant, i.e., greater in absolute value than some prescribed small threshold value $\epsilon$. We identify the index set $J_{\bm{a}}=\{j: |a_{j}|>\epsilon \}$ and set $P_{\bm{a}}=|J_{\bm{a}}|$. Similarly, reduced basis index sets for all MW expansions involved are identified at a cost that is linear in the order of MW expansion. The vectors $\bm{a}^{P}$, $\bm{b}^{P}$, and $\bm{c}^{P}$ are thus replaced by $\bm{a}^{P_{\bm{a}}}$, $\bm{b}^{P_{\bm{b}}}$, and $\bm{c}^{P_{\bm{c}}}$ where $P_{\bm{a}}, P_{\bm{b}}, P_{\bm{c}} \leq P$. At a reduced cost, we may now approximate the matrices $\bm{A}$ and $\bm{B}$,
\[
[\bm{A}(\bm{a}^{P})]_{jk} \approx [\tilde{\bm{A}}(\bm{a}^{P_{\bm{a}}})]_{jk} \equiv \sum_{i \in J_{\bm{a}}} \left\langle \psi_{i}\psi_{j}\psi_{k} \right\rangle a_{i},  \quad j,k =1,\dots, P,
\]
\[
[\bm{B}(\bm{a}^{P}, \bm{b}^{P})]_{jk} \approx [\tilde{\bm{B}}(\bm{a}^{P_{\bm{a}}}, \bm{b}^{P_{\bm{b}}})]_{jk} \equiv \sum_{h \in J_{\bm{a}}, i \in J_{\bm{b}}} \left\langle \psi_{h} \psi_{i}\psi_{j}\psi_{k} \right\rangle a_{h}b_{i},  \quad j,k =1,\dots, P.
\]
The matrices $\tilde{\bm{A}}(\bm{a}^{P_{\bm{a}}})$ and $\tilde{\bm{B}}(\bm{a}^{P_{\bm{a}}})$ are of the same size $P\times P$  as $\bm{A}(\bm{a}^{P})$ and $\bm{B}(\bm{a}^{P}, \bm{b}^{P})$, but they are formed using fewer floating-points operations.
These matrices are used to compute stochastic products via the matrix-vector products 
\begin{align}
\bm{A}(\bm{a}^{P})\bm{b}^{P} &\approx \tilde{\bm{A}}(\bm{a}^{P_{\bm{a}}})\bm{b}^{P}, \label{eq:A_red} \\
\bm{B}(\bm{a}^{P},\bm{b}^{P})\bm{c}^{P} &\approx \tilde{\bm{B}}(\bm{a}^{P_{\bm{a}}},\bm{b}^{P_{\bm{b}}})\bm{c}^{P}.\label{eq:B_red}
\end{align}
Further cost reduction in the evaluations of~\eqref{eq:A_red}-\eqref{eq:B_red} is possible by using the sparsity in $\bm{b}^{P}$ in~\eqref{eq:A_red} and the sparsity in $\bm{c}^{P}$  in~\eqref{eq:B_red}. In~\eqref{eq:A_red}, non-zero columns of $\tilde{\bm{A}}(\bm{a}^{P_{\bm{a}}})$ that are multiplied with almost-zero entries of $\bm{b}^{P}$ have negligible contribution to the matrix-vector product and can be omitted without loss of accuracy. Similarly, column vectors of $\tilde{\bm{B}}(\bm{a}^{P_{\bm{a}}},\bm{b}^{P_{\bm{b}}})$ that are multiplied by almost-zero entries of $\bm{c}^{P}$ in~\eqref{eq:B_red} can also be omitted. This is achieved by introducing reduced-size matrices
\begin{align}
 \tilde{\bm{A}}(\bm{a}^{P_{\bm{a}}})\bm{b}^{P}  &\approx \hat{\bm{A}}(\bm{a}^{P_{\bm{a}}},P_{\bm{b}})\bm{b}^{P_{\bm{b}}} \label{eq:A_red_red}, \\
 \tilde{\bm{B}}(\bm{a}^{P_{\bm{a}}},\bm{b}^{P_{\bm{b}}})\bm{c}^{P} &\approx \hat{\bm{B}}(\bm{a}^{P_{\bm{a}}},\bm{b}^{P_{\bm{b}}},P_{\bm{c}})\bm{c}^{P_{\bm{c}}},
\end{align}
where $\hat{\bm{A}}(\bm{a}^{P_{\bm{a}}},P_{\bm{b}}) \in \mathbbm{R}^{P \times P_{\bm{b}}}$ and $\hat{\bm{B}}(\bm{a}^{P_{\bm{a}}},\bm{b}^{P_{\bm{b}}},P_{\bm{c}}) \in \mathbbm{R}^{P \times P_{\bm{c}}}$. 
The construction of the reduced-cost matrix-vector products is systematic and efficient since it only involves a direct (linear) search of the vectors $\bm{a}^{P}$, $\bm{b}^{P}$, and $\bm{c}^{P}$ to find the index sets $J_{\bm{a}}$, $J_{\bm{b}}$, and $J_{\bm{c}}$. There is no need to search the matrices for negligible elements. It is possible to improve on the cost of searching the vectors for significant MW entries by exploiting the hierarchical structures of wavelets. If the search is performed from coarse levels to fine levels of stochastic partitioning and a wavelet is found to have insignificant impact, then one may terminate the  procedure since the descendants of that wavelet will also be negligible. This strategy is of particular interest in stochastic representations with a large number of wavelet levels but will not be further considered in this work.

Note that due to the space-time localization of hyperbolic problems (traveling waves), by setting a very small threshold $\epsilon$ the approximation is essentially exact, but we nevertheless obtain a substantial cost reduction by avoiding costly operations that have almost zero contribution. This strategy is pursued in the numerical experiments and no reduction in accuracy is observed - only in computational time - by using the locally reduced-order basis.

\subsection{Discretization of the Stochastic Galerkin System}
The reduced-order stochastic Galerkin method can be implemented with any numerical method that can capture the highly discontinuous solution of the extended nonlinear system of hyperbolic PDEs. We use a Riemann solver with the HLL (after Harten, Lax and van Leer) flux introduced in \citep{Harten_etal_83} and further developed by \citep{Einfeld:1988}. The HLL solver relies on estimates of the fastest velocities, $\sigma_L$ and $\sigma_R$, i.e. the estimated minimum and maximum eigenvalues of the Jacobian $\bm{J}^{P}=\partial \bm{f}^{P}/\partial \bm{S}^{P}$ of the flux. Kurganov et al. presented a class of central-upwind schemes where their second-order accurate polynomial reconstruction is very similar to the HLL flux \citep{Kurganov_etal_01}. The numerical flux function we use can thus also be labeled a central-upwind flux.

At the interface between the cells $j$ and $j+1$ in a single dimension, the vector valued HLL flux is defined by
\begin{equation}
\label{eq:hll_flux}
F_{j+\frac{1}{2}}=\left\{
\begin{array}{ll}
\bm{f}^{P} \left(\bm{S}^L_{j+\frac{1}{2}} \right) & \mbox{if } \sigma_L \geq 0\\
\\
\frac{\sigma_R \bm{f}^{P} \left(\bm{S}^L_{j+\frac{1}{2}} \right)-\sigma_L \bm{f}^{P} \left(\bm{S}^R_{j+\frac{1}{2}} \right)+\sigma_L \sigma_R \left(\bm{S}^R_{j+\frac{1}{2}}-\bm{S}^L_{j+\frac{1}{2}}\right)}{\sigma_R-\sigma_L} & \mbox{if }  \sigma_L < 0 < \sigma_R\\
\\
\bm{f}^{P}\left(\bm{S}^R_{j+\frac{1}{2}}\right) & \mbox{if }  S_R \leq 0
\end{array}
\right.,
\end{equation}
where $\bm{S}^L_{j+\frac{1}{2}}$ and $\bm{S}^R_{j+\frac{1}{2}}$ denote flux-limited left and right states, respectively. In the numerical experiments, we use the the minmod limiter since more diffusive flux limiters might lead to failure in resolving composite waves for non-convex flux functions~\citep{Kurganov_etal_07}. In general, obtaining accurate eigenvalue estimates may be computationally costly.
Note that in the deterministic case, we expect $f'(S)\geq 0$, and the scheme becomes an upwind scheme, since this implies $\sigma_L \geq 0$. However, this is not necessarily the case in the stochastic Galerkin setting.

The generalization to two dimensions coincides with the dimension-by-dimension application of the one-dimensional flux function. It should be noted, however, that the derivation of the scheme is genuinely multi-dimensional. It has been proven to satisfy a discrete maximum principle subject to a time-step restriction in two dimensions together with the minmod flux limiter \citep{Kurganov_etal_01}. This is an important property due to the fact that in many cases dimension-by-dimension treatment of the flux limiters lead to spurious oscillations of the solution \citep{Suresh_00} and we do indeed encounter them when the discrete maximum principle in \citep{Kurganov_etal_01} is not fulfilled. The time-step restriction due to the discrete maximum principle is more strict than the CFL (Courant-Friedrichs-Lewy) condition that is necessary but not sufficient to ensure stability for the Runge-Kutta method used for the time integration.


\section{Numerical Results}
\label{sec:num_res}

\subsection{Single Spatial Dimension and Stochastic Velocity}
To study numerical convergence, we first present an example in one spatial dimension and a single stochastic variable. Let the velocity $u$ be spatially constant (consistent with the incompressibility assumption) and uniformly distributed, $u\in \mathcal{U}[u_{min}, u_{max}]$, $u_{min}=0.8,u_{max}=1.2$, and consider a Riemann problem with $S=1$ for $x\leq 0$, $S=0$ for $x>0$. 
Since $u$ is random, $S$ is also random. However, the randomness in $u$ does not affect the value $S_*$ of the water saturation at the shock. Therefore, the shock location of the stochastic problem can be found by analysis of the characteristics,
\[
x_{shock} = x_{0}+ u(\xi)f'(S_{*})t,
\]
where $x_0=0$ is the initial location of the discontinuity. Since the velocity $u$ is restricted to a finite interval, we expect shocks of the realizations of the solution to be confined to the interval $[u_{min} f'(S_*)t, u_{max}f'(S_*)t]$. The stochastic Galerkin approximation of the original stochastic problem should reflect this fact. Figure \ref{fig:bl_1dl_P_2_4_8} depicts the mean value for different orders of piecewise constant wavelets, $P=4$ (left), $P=8$ (middle) and $P=16$ (right). 
The multiple discontinuities are located within the interval $[0.022, 0.033]$ corresponding to the spread of the shocks of the original problem, as seen in Figure \ref{fig:bl_1dl_mean_and_var}.

\begin{figure}[H]
\centering
\subfigure[$P=4$]
{\includegraphics[width=0.32\textwidth]{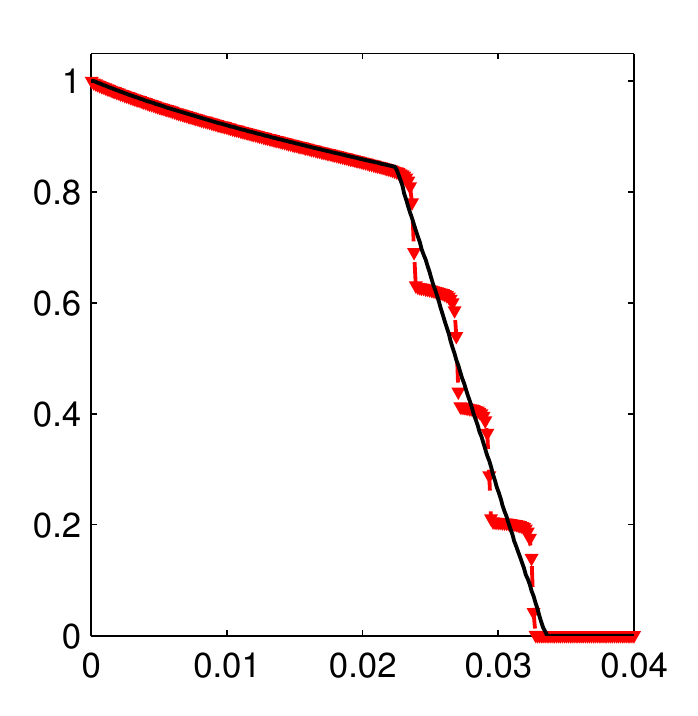}}
\subfigure[$P=8$]
{\includegraphics[width=0.32\textwidth]{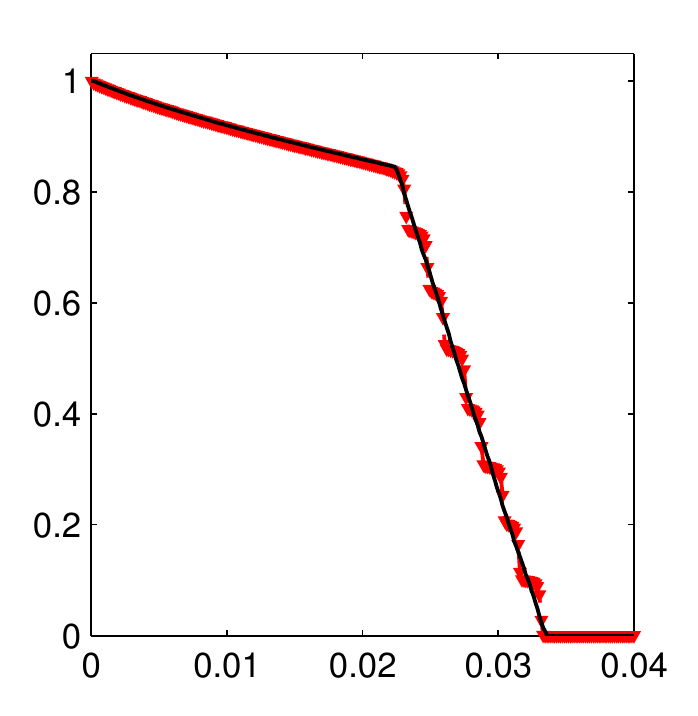}}
\subfigure[$P=16$]
{\includegraphics[width=0.32\textwidth]{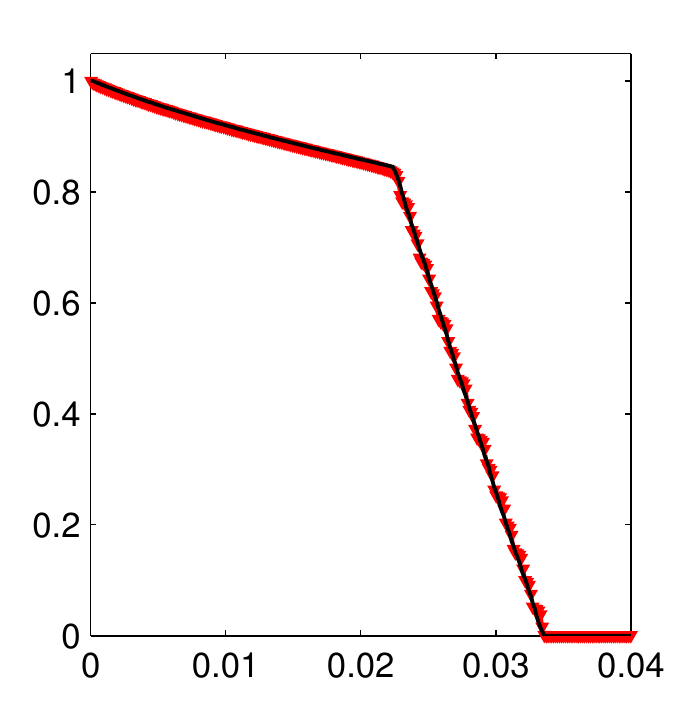}}
	\caption{Mean value of $S$ at $t=0.025$, 300 spatial points, $a=2$. The black curve is the reference Monte Carlo solution.}
	\label{fig:bl_1dl_P_2_4_8}
\end{figure}

\begin{figure}[H]
\centering
\subfigure[Mean values and standard deviations for different orders of MW expansion ($P$).]
{\includegraphics[width=0.48\textwidth]{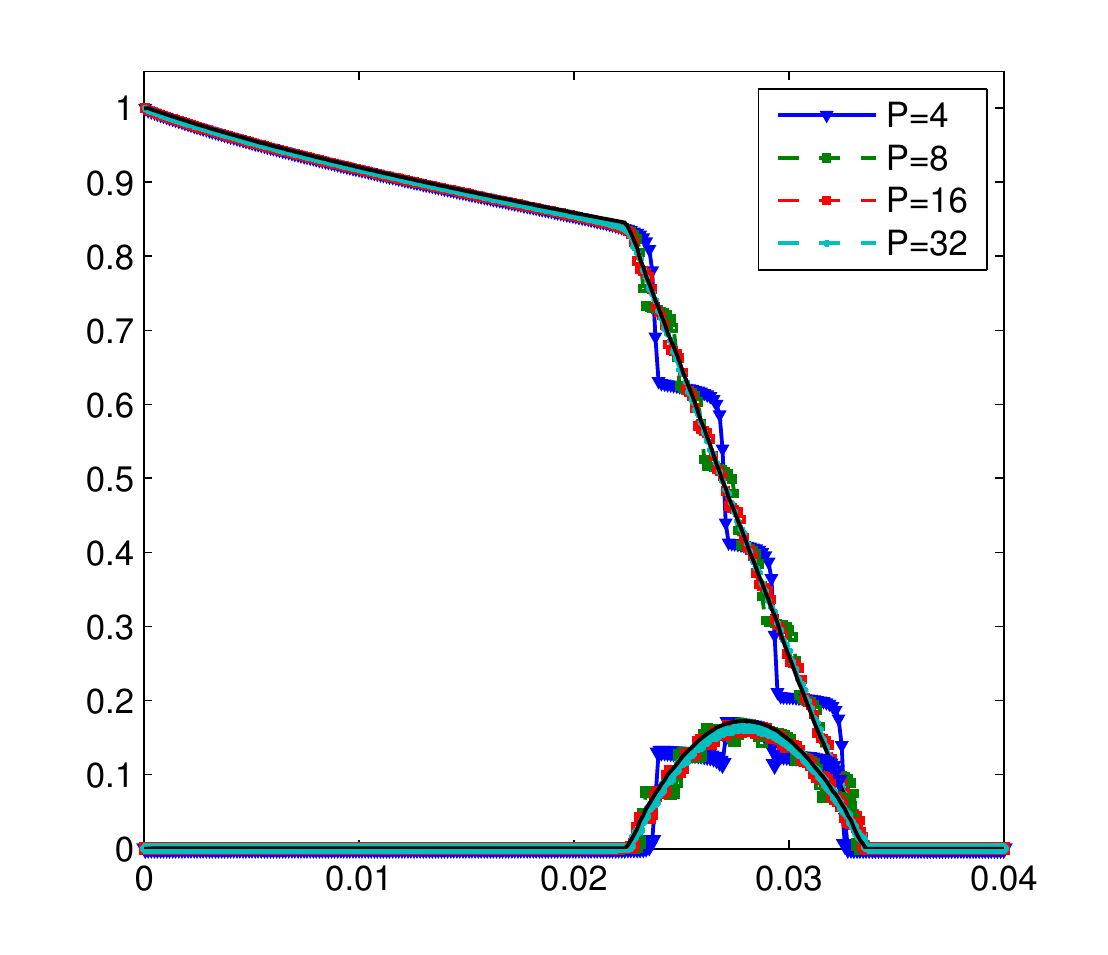}}
\subfigure[Zoom of left figure.]
{\includegraphics[width=0.48\textwidth]{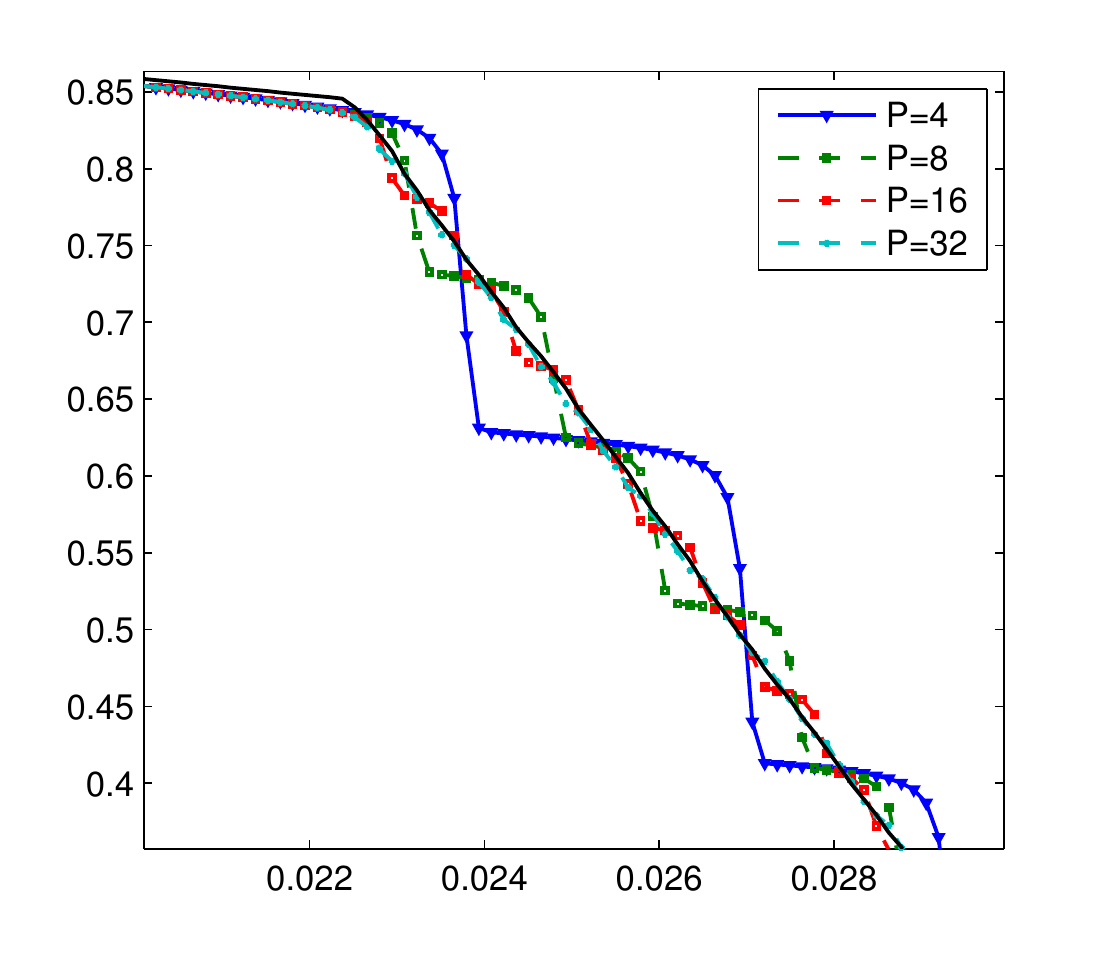}}
	\caption{Mean value and standard deviation of $S$ at $t=0.025$, 300 spatial points, $a=2$. Piecewise linear  wavelets. Also shown is the reference Monte Carlo solution (black).}
	\label{fig:bl_1dl_mean_and_var}
\end{figure}

The numerical convergence in comparison to a Monte Carlo reference solution of the single-dimensional case was investigated for various values of the viscosity ratio in \citep{Pettersson_Tchelepi_14}. It was found that multiwavelets with relatively low-order piecewise polynomial were appropriate to represent the solution. The piecewise linear and quadratic wavelets employed in the numerical results below do not exhibit the oscillations around discontinuities observed when using higher order polynomial representation. Such oscillations are important to avoid since they may lead to unphysical solution values (e.g. negative saturation) and breakdown of the numerical method.

The MW approximation deteriorates over time due to the growth of the higher order MW coefficients which makes the stochastic truncation error more severe with time~\citep{Wan_Karniadakis_05_b}. An interesting phenomenon related to this fact has been noted in the numerical results. If the stochastic truncation error grows sufficiently large, the lower-order MW coefficients grow in unexpected ways. This is shown in Figure~\ref{fig:basis_comp_1D} (a) where a $P=6$ ($N_p=2, N_r=1$) order wavelet basis has been used. The HLL solver yields a mean solution that is not monotone and therefore evidently a poor approximation. For comparison, the solution using the more diffusive Lax-Friedrichs flux with minmod flux limiters yields a monotone solution and might appear as a better choice in this case. However, this flux function can never capture the true solution unless the mesh is excessively fine. In contrast, if the order of wavelet refinement is increased as in Figure~\ref{fig:basis_comp_1D} (b) where a $P=24$ ($N_p=2, N_r=3$) order basis is used, the HLL solver performs well and captures the statistics of the solution. In addition, with a MW basis with sufficient piecewise polynomial order, the discontinuities seen in Figure~\ref{fig:bl_1dl_mean_and_var} are no longer visible in Figure~\ref{fig:basis_comp_1D} (b). We will therefore use the HLL solver in the rest of the manuscript and make sure that the stochastic basis employed is rich enough not to result in unphysical solutions similar to those shown in Figure~\ref{fig:basis_comp_1D} (a).

\begin{figure}[H]
\centering
\subfigure[ $N_{p}=2$.]
{\includegraphics[width=0.48\textwidth]{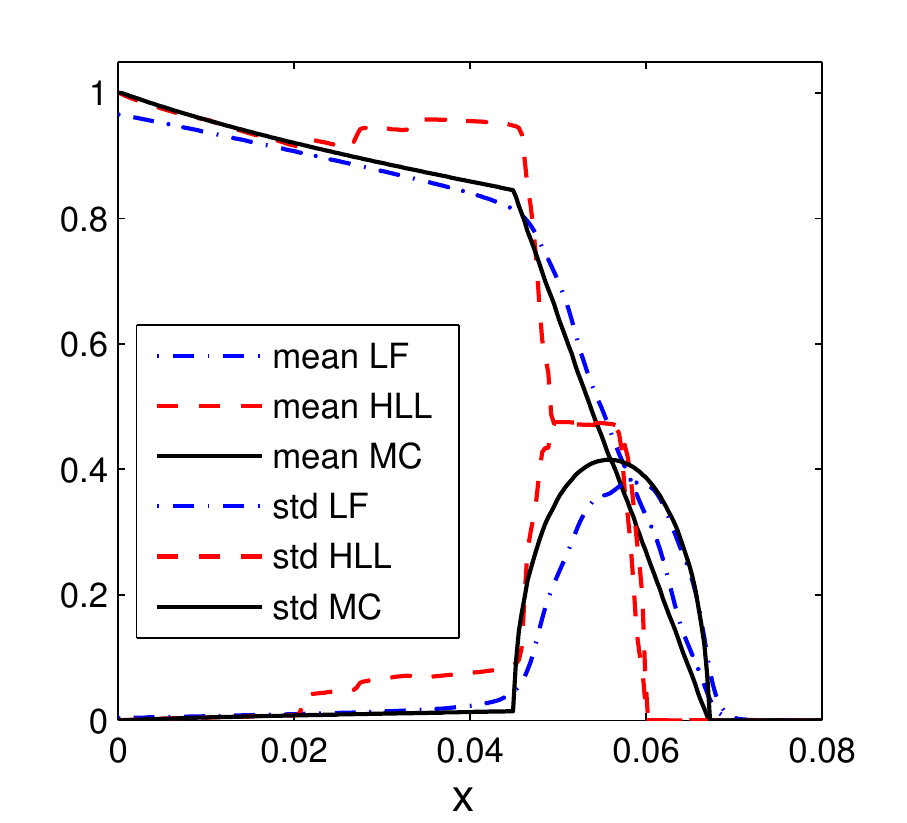}}
\subfigure[$N_{p}=3$.]
{\includegraphics[width=0.48\textwidth]{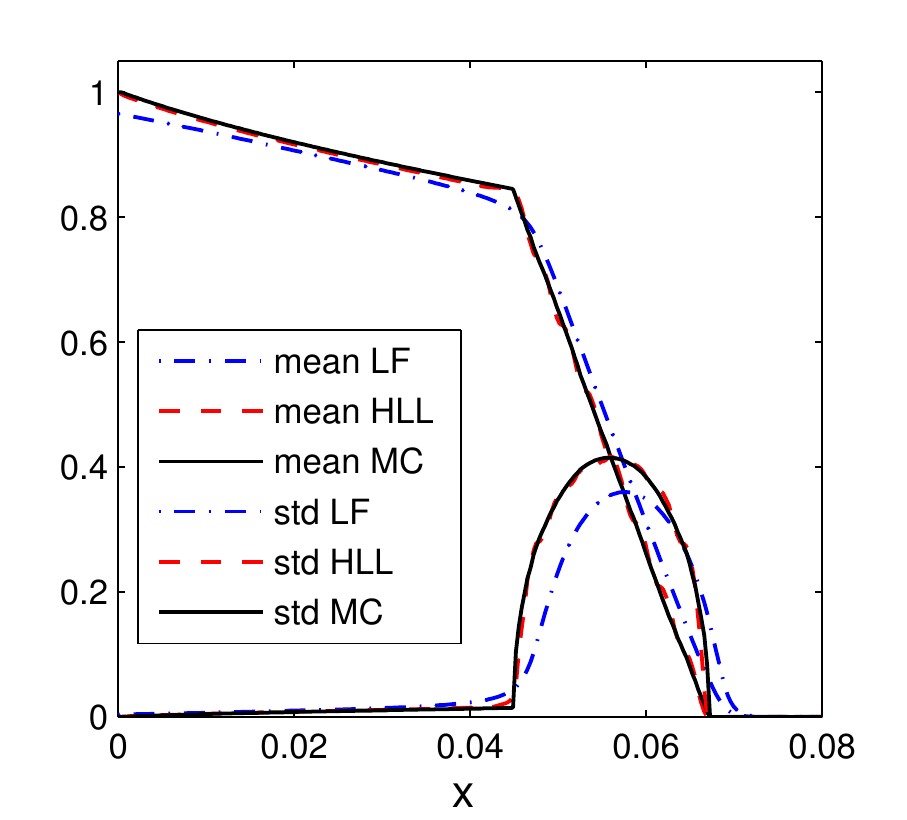}}
	\caption{The mean and standard deviation (std) of the sauration at time 0.05 for different stochastic bases and for the HLL and Lax-Friedrichs (LF) flux, respectively.}
	\label{fig:basis_comp_1D}
\end{figure}


The system size grows linearly in the MW expansion order $P$ and for practical applications it is essential to alleviate the computational cost. Figure~\ref{fig:temp_comp_1D} depicts the simulation times as a function of the order of MW expansion. Piecewise linear and piecewise quadratic wavelets are used as basis functions and we compare the solution of the full system with the adaptive reduced-order model described in Section~\ref{sec:adapt_red_order}. For large $P$, the gain in computational cost is more than an order of magnitude. 
\begin{figure}[H]
\centering
\subfigure[ $N_{p}=2$.]
{\includegraphics[width=0.48\textwidth]{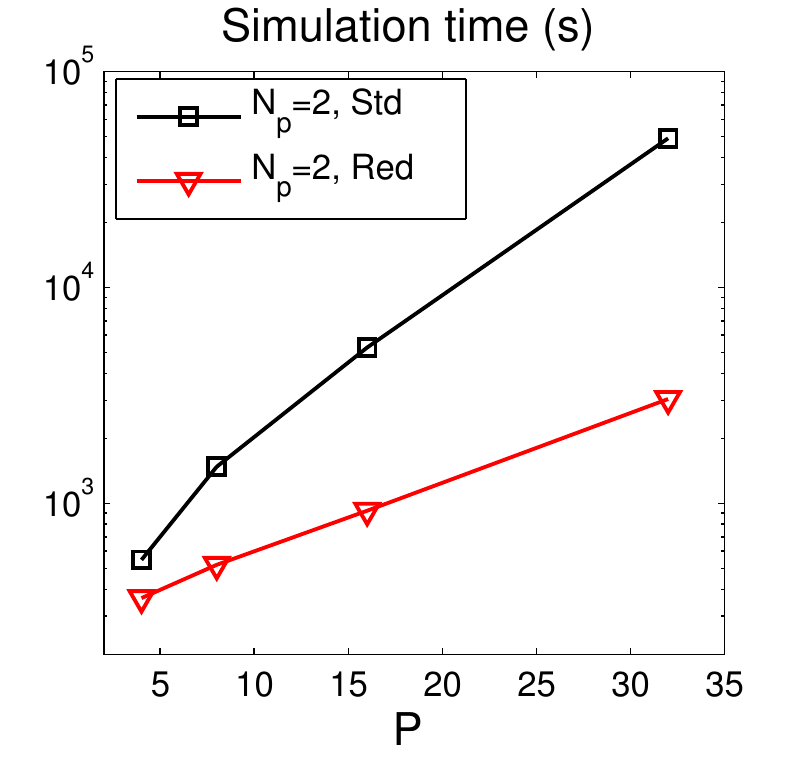}}
\subfigure[$N_{p}=3$.]
{\includegraphics[width=0.48\textwidth]{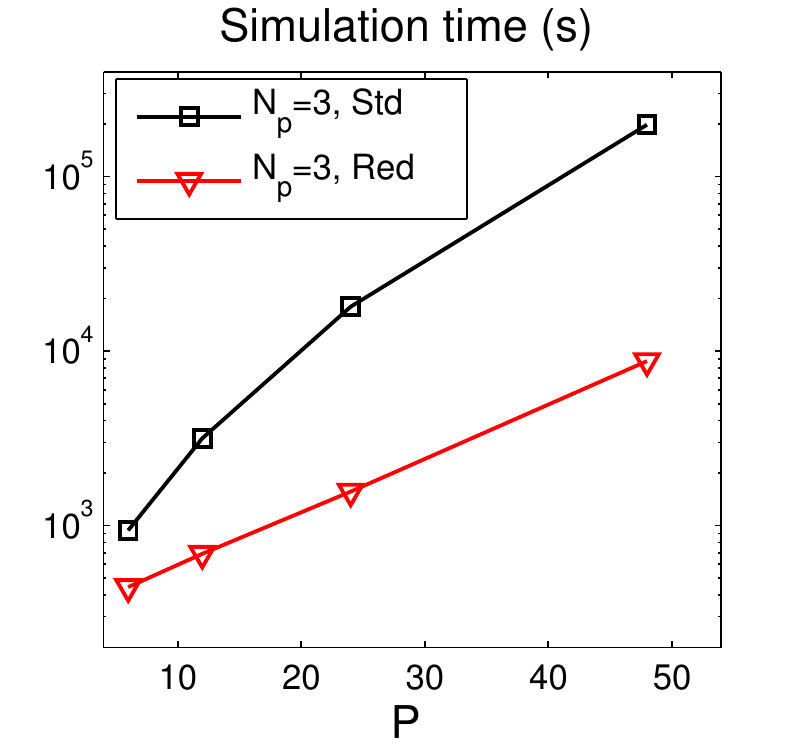}}
	\caption{Simulation times for the standard full-order (Std) and the reduced-order (Red) stochastic Galerkin formulations as a function of the total order of MW representation.}
	\label{fig:temp_comp_1D}
\end{figure}



\setlength{\unitlength}{0.032cm}

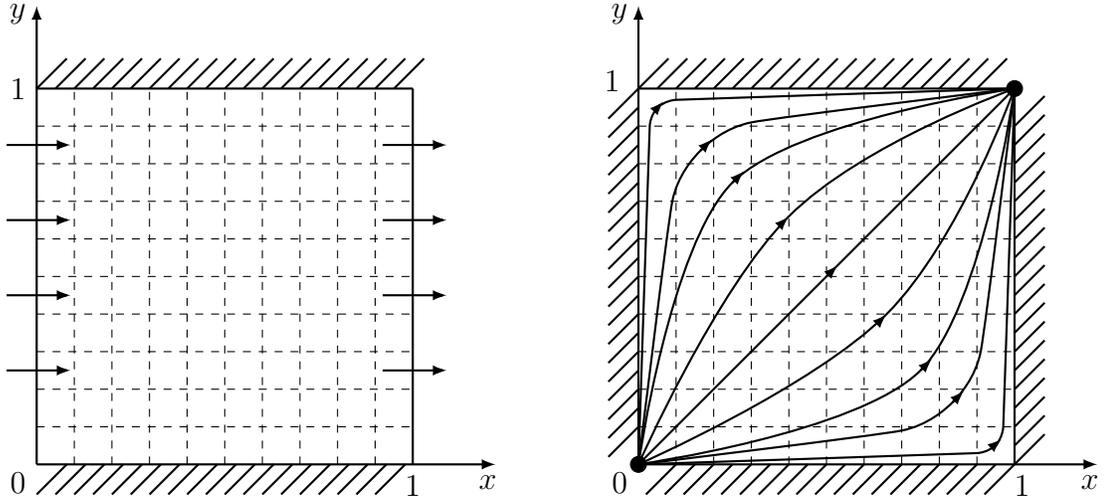
\begin{figure}[H] 
\begin{center}
\begin{tikzpicture}


\draw[thick] (0,0) -- (0,6);
\draw [-latex, black, thick, shorten >= 0.90cm] (0,6) -- (0,7);

\draw[thick] (0,0) -- (6,0);
\draw [-latex, black, thick, shorten >= 0.90cm] (6,0) -- (7,0);

\draw[thin, dashed] (0.5,0) -- (0.5,5);
\draw[thin, dashed] (1.0,0) -- (1.0,5);
\draw[thin, dashed] (1.5,0) -- (1.5,5);
\draw[thin, dashed] (2.0,0) -- (2.0,5);
\draw[thin, dashed] (2.5,0) -- (2.5,5);
\draw[thin, dashed] (3.0,0) -- (3.0,5);
\draw[thin, dashed] (3.5,0) -- (3.5,5);
\draw[thin, dashed] (4.0,0) -- (4.0,5);
\draw[thin, dashed] (4.5,0) -- (4.5,5);
\draw[thick] (5,0) -- (5,5);

\draw[thin, dashed] (0,0.5) -- (5,0.5);
\draw[thin, dashed] (0,1.0) -- (5,1.0);
\draw[thin, dashed] (0,1.5) -- (5,1.5);
\draw[thin, dashed] (0,2.0) -- (5,2.0);
\draw[thin, dashed] (0,2.5) -- (5,2.5);
\draw[thin, dashed] (0,3.0) -- (5,3.0);
\draw[thin, dashed] (0,3.5) -- (5,3.5);
\draw[thin, dashed] (0,4.0) -- (5,4.0);
\draw[thin, dashed] (0,4.5) -- (5,4.5);
\draw[thick] (0,5) -- (5,5);

\node [below, left] at (0,-0.25) {$0$};
\node [below] at (5,0) {$1$};
\node [left] at (0,5) {$1$};
\node [below] at (6,0) {$x$};
\node [left] at (0,6) {$y$};

\draw[thick] (-0.4,1.25) -- (0.4,1.25);
\draw [-latex, black, thick, shorten >= 0.95cm] (0.4,1.25) -- (1.4,1.25);
\draw[thick] (-0.4,2.25) -- (0.4,2.25);
\draw [-latex, black, thick, shorten >= 0.95cm] (0.4,2.25) -- (1.4,2.25);
\draw[thick] (-0.4,3.25) -- (0.4,3.25);
\draw [-latex, black, thick, shorten >= 0.95cm] (0.4,3.25) -- (1.4,3.25);
\draw[thick] (-0.4,4.25) -- (0.4,4.25);
\draw [-latex, black, thick, shorten >= 0.95cm] (0.4,4.25) -- (1.4,4.25);
\draw[thick] (4.6,1.25) -- (5.4,1.25);
\draw [-latex, black, thick, shorten >= 0.95cm] (5.4,1.25) -- (6.4,1.25);
\draw[thick] (4.6,2.25) -- (5.4,2.25);
\draw [-latex, black, thick, shorten >= 0.95cm] (5.4,2.25) -- (6.4,2.25);
\draw[thick] (4.6,3.25) -- (5.4,3.25);
\draw [-latex, black, thick, shorten >= 0.95cm] (5.4,3.25) -- (6.4,3.25);
\draw[thick] (4.6,4.25) -- (5.4,4.25);
\draw [-latex, black, thick, shorten >= 0.95cm] (5.4,4.25) -- (6.4,4.25);

%
%
\draw[thick] (0.0,5) -- (0.4,5.4);
\draw[thick] (0.25,5) -- (0.65,5.4);
\draw[thick] (0.5,5) -- (0.9,5.4);
\draw[thick] (0.75,5) -- (1.15,5.4);
\draw[thick] (1.0,5) -- (1.4,5.4);
\draw[thick] (1.25,5) -- (1.65,5.4);
\draw[thick] (1.5,5) -- (1.9,5.4);
\draw[thick] (1.75,5) -- (2.15,5.4);
\draw[thick] (2.0,5) -- (2.4,5.4);
\draw[thick] (2.25,5) -- (2.65,5.4);
\draw[thick] (2.5,5) -- (2.9,5.4);
\draw[thick] (2.75,5) -- (3.15,5.4);
\draw[thick] (3.0,5) -- (3.4,5.4);
\draw[thick] (3.25,5) -- (3.65,5.4);
\draw[thick] (3.5,5) -- (3.9,5.4);
\draw[thick] (3.75,5) -- (4.15,5.4);
\draw[thick] (4.0,5) -- (4.4,5.4);
\draw[thick] (4.25,5) -- (4.65,5.4);
\draw[thick] (4.5,5) -- (4.9,5.4);
\draw[thick] (4.75,5) -- (5.15,5.4);

%
%
\draw[thick] (0.0,-0.4) -- (0.4,0);
\draw[thick] (0.25,-0.4) -- (0.65,0);
\draw[thick] (0.5,-0.4) -- (0.9,0);
\draw[thick] (0.75,-0.4) -- (1.15,0);
\draw[thick] (1.0,-0.4) -- (1.4,0);
\draw[thick] (1.25,-0.4) -- (1.65,0);
\draw[thick] (1.5,-0.4) -- (1.9,0);
\draw[thick] (1.75,-0.4) -- (2.15,0);
\draw[thick] (2.0,-0.4) -- (2.4,0);
\draw[thick] (2.25,-0.4) -- (2.65,0);
\draw[thick] (2.5,-0.4) -- (2.9,0);
\draw[thick] (2.75,-0.4) -- (3.15,0);
\draw[thick] (3.0,-0.4) -- (3.4,0);
\draw[thick] (3.25,-0.4) -- (3.65,0);
\draw[thick] (3.5,-0.4) -- (3.9,0);
\draw[thick] (3.75,-0.4) -- (4.15,0);
\draw[thick] (4.0,-0.4) -- (4.4,0);
\draw[thick] (4.25,-0.4) -- (4.65,0);
\draw[thick] (4.5,-0.4) -- (4.9,0);

%
%


\draw[thick] (8,0) -- (8,6);
\draw [-latex, black, thick, shorten >= 0.90cm] (8,6) -- (8,7);

\draw[thick] (8,0) -- (14,0);
\draw [-latex, black, thick, shorten >= 0.90cm] (14,0) -- (15,0);

\draw[thin, dashed] (8.5,0) -- (8.5,5);
\draw[thin, dashed] (9.0,0) -- (9.0,5);
\draw[thin, dashed] (9.5,0) -- (9.5,5);
\draw[thin, dashed] (10.0,0) -- (10.0,5);
\draw[thin, dashed] (10.5,0) -- (10.5,5);
\draw[thin, dashed] (11.0,0) -- (11.0,5);
\draw[thin, dashed] (11.5,0) -- (11.5,5);
\draw[thin, dashed] (12.0,0) -- (12.0,5);
\draw[thin, dashed] (12.5,0) -- (12.5,5);
\draw[thick] (13,0) -- (13,5);

\draw[thin, dashed] (8,0.5) -- (13,0.5);
\draw[thin, dashed] (8,1.0) -- (13,1.0);
\draw[thin, dashed] (8,1.5) -- (13,1.5);
\draw[thin, dashed] (8,2.0) -- (13,2.0);
\draw[thin, dashed] (8,2.5) -- (13,2.5);
\draw[thin, dashed] (8,3.0) -- (13,3.0);
\draw[thin, dashed] (8,3.5) -- (13,3.5);
\draw[thin, dashed] (8,4.0) -- (13,4.0);
\draw[thin, dashed] (8,4.5) -- (13,4.5);
\draw[thick] (8,5) -- (13,5);

\node [below, left] at (8,-0.25) {$0$};
\node [below] at (13.1,0) {$1$};
\node [below] at (14,0) {$x$};
\node [left] at (7.9,5.1) {$1$};
\node [left] at (8,6) {$y$};


\draw[thick] plot [smooth,tension=1.5] coordinates{(8,0) (10.5,2.5)}; 
\draw[thick] plot [smooth,tension=1.5] coordinates{(10.5,2.5) (13,5)};
\draw [-latex, black, thick, shorten >= 0.50cm] (10.5,2.5) -- (11,3);

\draw[thick] plot [smooth,tension=0.7] coordinates{(8,0) (11.3,2) (13,5)};
\draw[thick] plot [smooth,tension=0.7] coordinates{(8,0) (10,3.3) (13,5)};
\draw [-latex, black, thick, shorten >= 0.15cm] (11.2,1.9) -- (11.4,2.1);
\draw [-latex, black, thick, shorten >= 0.15cm] (9.9,3.2) -- (10.1,3.4);

\draw[thick] plot [smooth,tension=0.6] coordinates{(8,0) (11.75,1.25) (13,5)};
\draw[thick] plot [smooth,tension=0.6] coordinates{(8,0) (9.25,3.75) (13,5)};
\draw [-latex, black, thick, shorten >= 0.00cm] (11.75,1.25) -- (11.90,1.40);
\draw [-latex, black, thick, shorten >= 0.00cm] (9.25,3.75) -- (9.40,3.90);

\draw[thick] plot [smooth,tension=0.3] coordinates{(8,0) (11.5,0.45) (12.175,0.825) (12.55,1.5) (13,5)};
\draw[thick] plot [smooth,tension=0.3] coordinates{(8,0) (8.45,3.5) (8.825,4.175) (9.5,4.55) (13,5)};
\draw [-latex, black, thick, shorten >= 0.00cm] (12.175,0.825) -- (12.325,0.975);
\draw [-latex, black, thick, shorten >= 0.00cm] (8.825,4.175) -- (8.975,4.325);

\draw[thick] plot [smooth,tension=0.1] coordinates{(8,0) (12.5,0.15) (12.75,0.25) (12.85,0.5) (13,5)};
\draw[thick] plot [smooth,tension=0.1] coordinates{(8,0) (8.15,4.5) (8.25,4.75) (8.5,4.85) (13,5)};
\draw [-latex, black, thick, shorten >= 0.10cm] (12.69,0.22) -- (12.90,0.40);
\draw [-latex, black, thick, shorten >= 0.10cm] (8.25,4.74) -- (8.40,4.90);

%
%
\draw[thick] (8.1,-0.4) -- (8.5,0);
\draw[thick] (8.35,-0.4) -- (8.75,0);
\draw[thick] (8.6,-0.4) -- (9.0,0);
\draw[thick] (8.85,-0.4) -- (9.25,0);
\draw[thick] (9.1,-0.4) -- (9.5,0);
\draw[thick] (9.35,-0.4) -- (9.75,0);
\draw[thick] (9.6,-0.4) -- (10.0,0);
\draw[thick] (9.85,-0.4) -- (10.25,0);
\draw[thick] (10.1,-0.4) -- (10.5,0);
\draw[thick] (10.35,-0.4) -- (10.75,0);
\draw[thick] (10.6,-0.4) -- (11.0,0);
\draw[thick] (10.85,-0.4) -- (11.25,0);
\draw[thick] (11.1,-0.4) -- (11.5,0);
\draw[thick] (11.35,-0.4) -- (11.75,0);
\draw[thick] (11.6,-0.4) -- (12.0,0);
\draw[thick] (11.85,-0.4) -- (12.25,0);
\draw[thick] (12.1,-0.4) -- (12.5,0);
\draw[thick] (12.35,-0.4) -- (12.75,0);
\draw[thick] (12.6,-0.4) -- (13.0,0);

%
\draw[thick] (13.0,0) -- (13.4,0.4);
\draw[thick] (13.0,0.25) -- (13.4,0.65);
\draw[thick] (13.0,0.5) -- (13.4,0.9);
\draw[thick] (13.0,0.75) -- (13.4,1.15);
\draw[thick] (13.0,1.0) -- (13.4,1.4);
\draw[thick] (13.0,1.25) -- (13.4,1.65);
\draw[thick] (13.0,1.5) -- (13.4,1.9);
\draw[thick] (13.0,1.75) -- (13.4,2.15);
\draw[thick] (13.0,2.0) -- (13.4,2.4);
\draw[thick] (13.0,2.25) -- (13.4,2.65);
\draw[thick] (13.0,2.5) -- (13.4,2.9);
\draw[thick] (13.0,2.75) -- (13.4,3.15);
\draw[thick] (13.0,3.0) -- (13.4,3.4);
\draw[thick] (13.0,3.25) -- (13.4,3.65);
\draw[thick] (13.0,3.5) -- (13.4,3.9);
\draw[thick] (13.0,3.75) -- (13.4,4.15);
\draw[thick] (13.0,4.0) -- (13.4,4.4);
\draw[thick] (13.0,4.25) -- (13.4,4.65);
\draw[thick] (13.0,4.5) -- (13.4,4.9);

%
%
\draw[thick] (8.0,5.0) -- (8.4,5.4);
\draw[thick] (8.25,5.0) -- (8.65,5.4);
\draw[thick] (8.5,5.0) -- (8.9,5.4);
\draw[thick] (8.75,5.0) -- (9.15,5.4);
\draw[thick] (9.0,5.0) -- (9.4,5.4);
\draw[thick] (9.25,5.0) -- (9.65,5.4);
\draw[thick] (9.5,5.0) -- (9.9,5.4);
\draw[thick] (9.75,5.0) -- (10.15,5.4);
\draw[thick] (10.0,5.0) -- (10.4,5.4);
\draw[thick] (10.25,5.0) -- (10.65,5.4);
\draw[thick] (10.5,5.0) -- (10.9,5.4);
\draw[thick] (10.75,5.0) -- (11.15,5.4);
\draw[thick] (11.0,5.0) -- (11.4,5.4);
\draw[thick] (11.25,5.0) -- (11.65,5.4);
\draw[thick] (11.5,5.0) -- (11.9,5.4);
\draw[thick] (11.75,5.0) -- (12.15,5.4);
\draw[thick] (12.0,5.0) -- (12.4,5.4);
\draw[thick] (12.25,5.0) -- (12.65,5.4);
\draw[thick] (12.5,5.0) -- (12.9,5.4);

\draw[thick] (7.6,0.1) -- (8.0,0.5);
\draw[thick] (7.6,0.35) -- (8.0,0.75);
\draw[thick] (7.6,0.6) -- (8.0,1.0);
\draw[thick] (7.6,0.85) -- (8.0,1.25);
\draw[thick] (7.6,1.1) -- (8.0,1.5);
\draw[thick] (7.6,1.35) -- (8.0,1.75);
\draw[thick] (7.6,1.6) -- (8.0,2.0);
\draw[thick] (7.6,1.85) -- (8.0,2.25);
\draw[thick] (7.6,2.1) -- (8.0,2.5);
\draw[thick] (7.6,2.35) -- (8.0,2.75);
\draw[thick] (7.6,2.6) -- (8.0,3.0);
\draw[thick] (7.6,2.85) -- (8.0,3.25);
\draw[thick] (7.6,3.1) -- (8.0,3.5);
\draw[thick] (7.6,3.35) -- (8.0,3.75);
\draw[thick] (7.6,3.6) -- (8.0,4.0);
\draw[thick] (7.6,3.85) -- (8.0,4.25);
\draw[thick] (7.6,4.1) -- (8.0,4.5);
\draw[thick] (7.6,4.35) -- (8.0,4.75);
\draw[thick] (7.6,4.6) -- (8.0,5.0);

\draw [fill=black] (13,5) circle (3pt);
\draw [fill=black] (8,0) circle (3pt);

\end{tikzpicture}
\end{center}
\caption{Schematic representation of the setup of the two dimensional test problems. The uniform mean flow domain has horizontal no-flow boundaries and vertical Dirichlet boundaries (left). The quarter five-spot problem has no-flow boundaries on all four sides, but a source (of water) at $(0,0)$ and a sink at $(1,1)$.}
\label{fig:setup}
\end{figure}

\subsection{Two Spatial Dimensions: Line injection}
\label{sec:num_res_rub90}
Consider uniform-in-the-mean flow aligned with the $x$-direction and injection along the line $(x,y)=(0,y)$, $0.25\leq y \leq 0.75$.
The setup is depicted in Figure~\ref{fig:setup} (left). For the velocity field, we use the 2D velocity covariance functions derived in \citep{Rubin_wrr_90} from a log-normal conductivity field. They are statistically homogeneous, i.e. $\mathbf{C}(\mathbf{x},\mathbf{x}')=\mathbf{\tilde{C}}(r_1,r_2)$ with $r_1=x'-x$, $r_2=y'-y$. For reference, the covariance matrix entries are plotted in Figure \ref{fig:cov_rub}. The color scales of the entries are different for visibility. The covariance function for the velocity is based on a first order perturbation approximation and was found to be valid for log-transmissivity variances $\sigma_Y^2 \leq 1$.  Note that this puts a restriction on the validity of the input statistics for this test case, but the stochastic transport solver iteself is not subject to this variance restriction. We believe that the velocity model is sufficiently accurate to demonstrate the representation properties of the KL-MW framework. Thus, for evaluation of the method, we assume that the covariance model represents the true velocity statistics. The velocity field in \citep{Rubin_wrr_90} is multi-variate Gaussian but we assume that it has a symmetric truncated multi-Gaussian distribution, i.e. the shape of the probability density function is essentially Gaussian but bounded to an interval that covers 99.7\% of the total probability. 
\begin{figure}[H]
\centering
\subfigure[ $C_{u_x,u_x}$.]
{\includegraphics[width=0.32\textwidth]{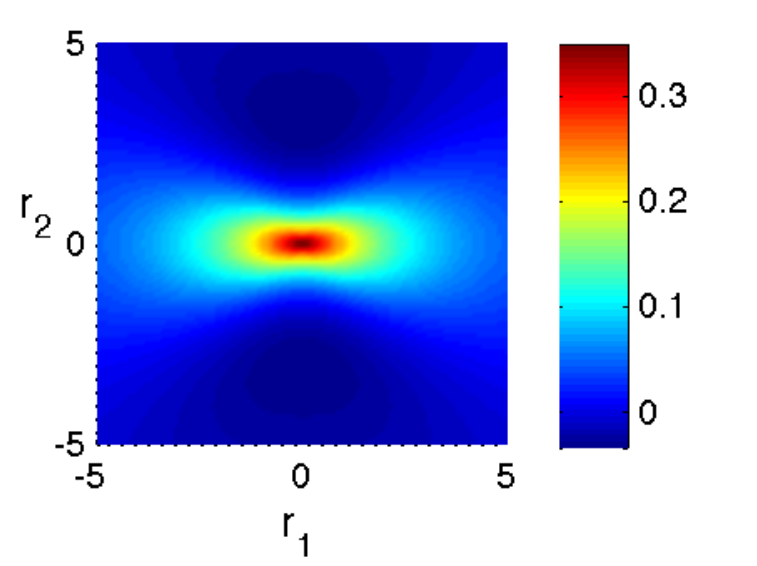}}
\subfigure[$C_{u_y,u_y}$.]
{\includegraphics[width=0.32\textwidth]{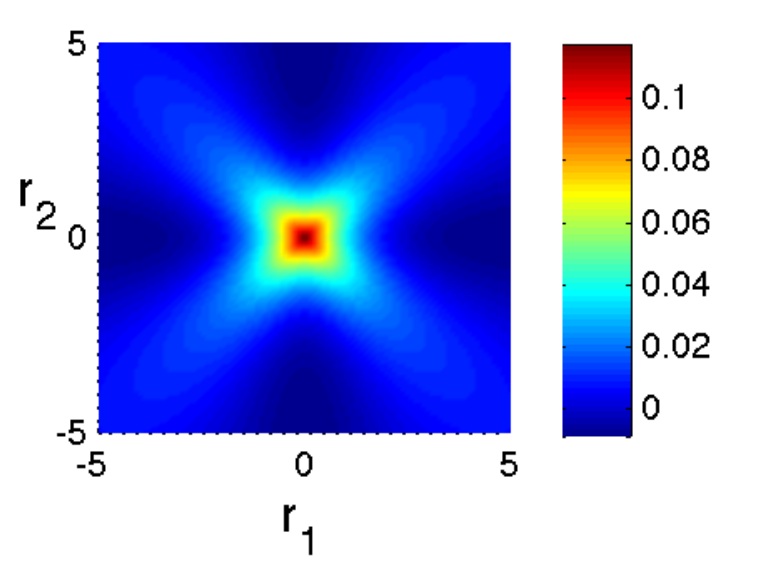}}
\subfigure[$C_{u_x,u_y}$.]
{\includegraphics[width=0.32\textwidth]{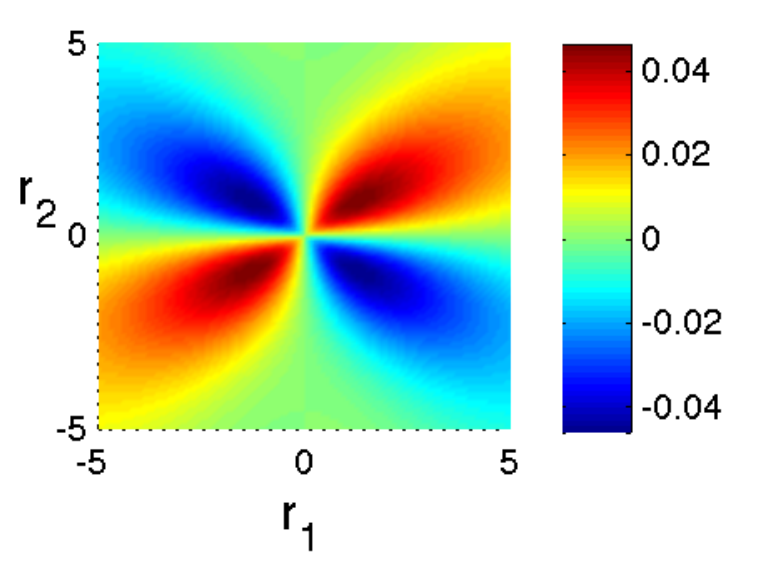}}
	\caption{Covariance matrix entries for the velocity field.}
	\label{fig:cov_rub}
\end{figure}
The solution of the generalized eigenvalue problem \eqref{eq:gen_evpr_kl} is in general computationally expensive and calls for sophisticated methods \citep{Schwab_Todor_06}. However, for this case of stationary covariance, we found that it was sufficient (in terms of speed and accuracy) to discretize \eqref{eq:gen_evpr_kl} with Simpson's rule, despite the lack of smoothness at $\br=\mathbf{0}$.

Consistently with Rubin, we assume no-flow conditions on the horizontal boundaries $(x,0)$ and $(x,1)$ and constant pressure gradient on the boundaries $(0,y)$ and $(1,y)$. For the transport problem, we impose deterministic injection $S=1$ at $x=0$ along the line $y \in (0.25,0.75)$. This condition is enforced weakly through a penalty term which is dependent on the stochastic horizontal velocity $u_x$. It is worthwhile to note that the treatment of this boundary condition is essential since the stochastic fluctuation of the transverse velocity $u_y$ would otherwise impact $S$ at $x=0$ over time. The stochastic fluctuation in $u_y$ leads to transport in the  $y$-direction at the injection boundary.

%

\begin{figure}[H]
\centering
\subfigure[$t=0.5$, multiwavelets.]
{\includegraphics[width=0.48\textwidth]{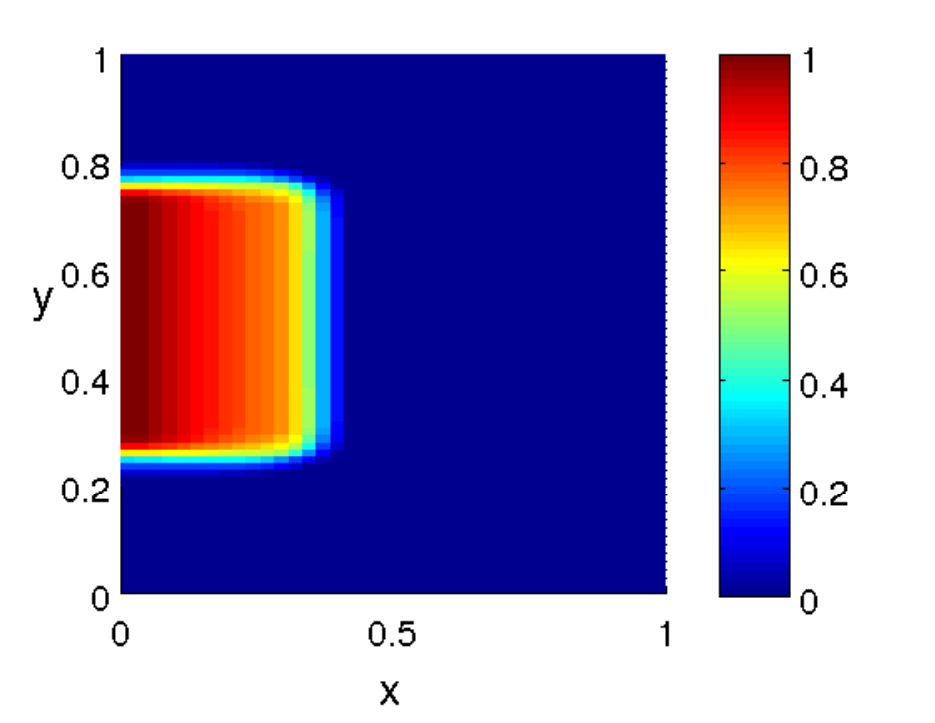}}
\subfigure[$t=0.5$, Monte Carlo.]
{\includegraphics[width=0.48\textwidth]{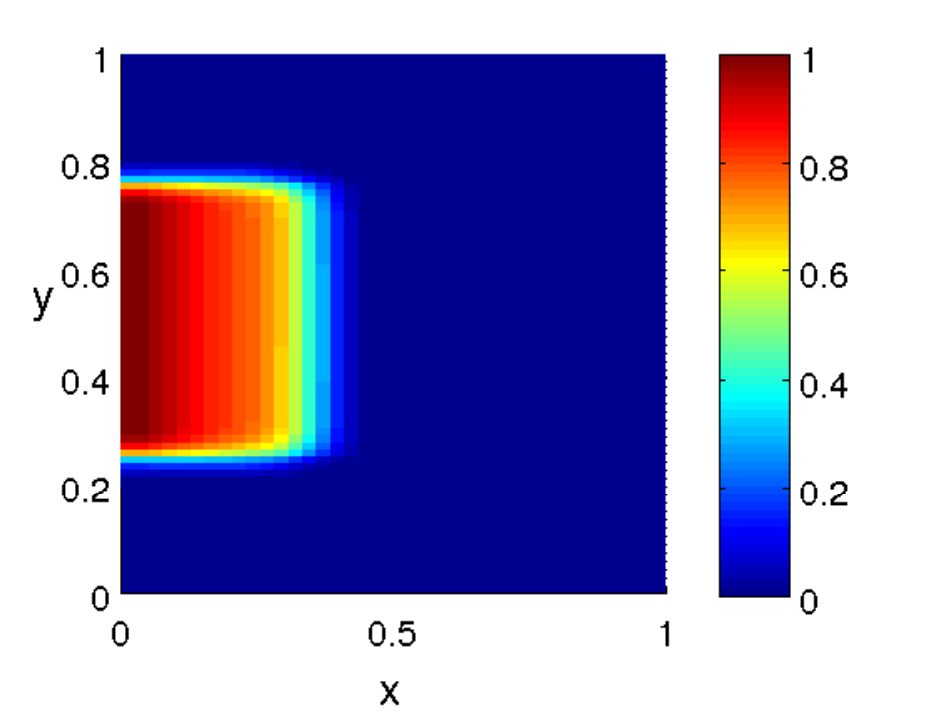}}
\subfigure[$t=1$, multiwavelets.]
{\includegraphics[width=0.48\textwidth]{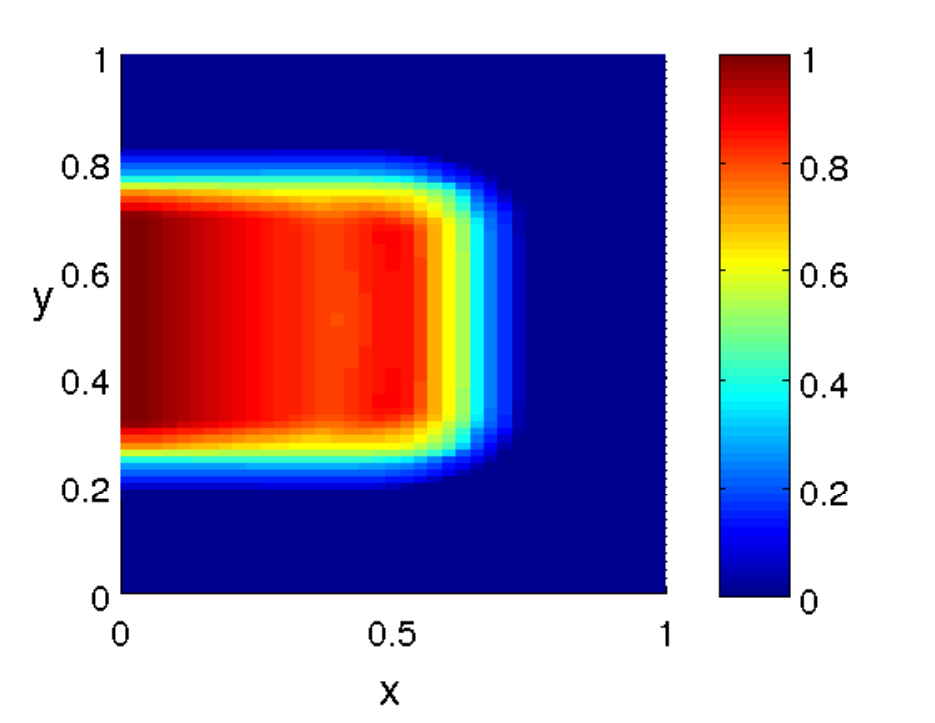}}
\subfigure[$t=1$, Monte Carlo.]
{\includegraphics[width=0.48\textwidth]{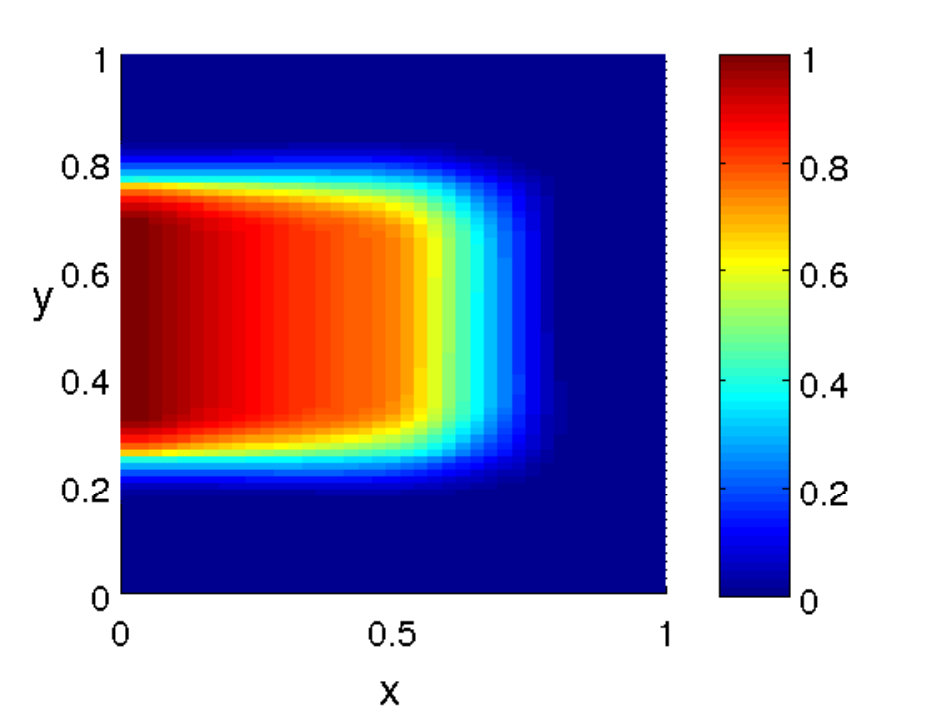}}
	\caption{Line injection. Mean value of the saturation, $t=0.25,0.5$.}
	\label{fig:line_mean}
\end{figure}

\begin{figure}[H]
\centering
\subfigure[$t=0.25$, multiwavelets.]
{\includegraphics[width=0.48\textwidth]{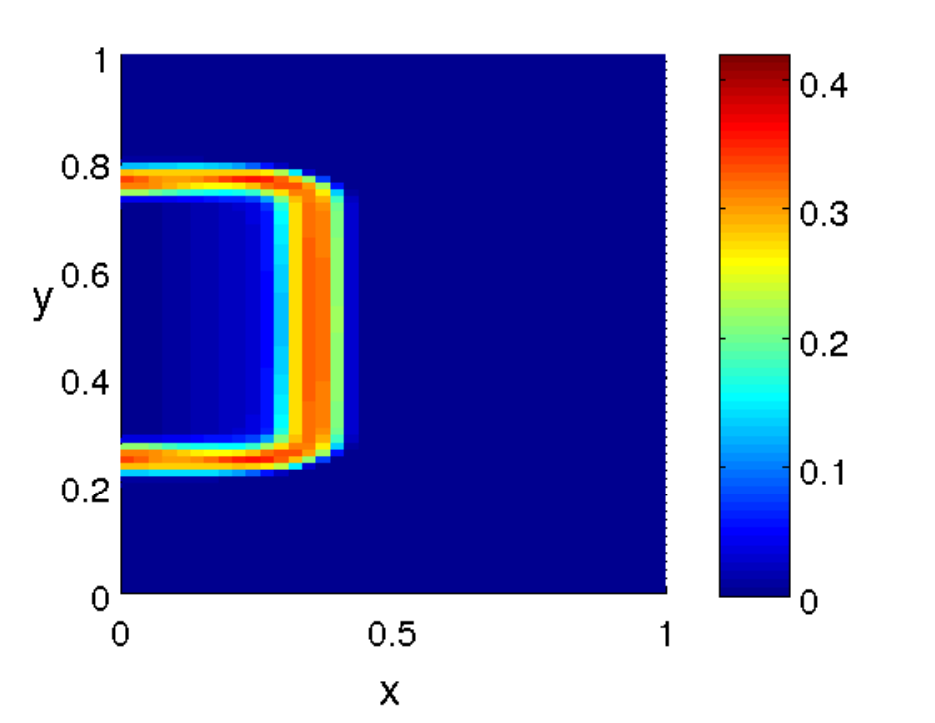}}
\subfigure[$t=0.25$, Monte Carlo.]
{\includegraphics[width=0.48\textwidth]{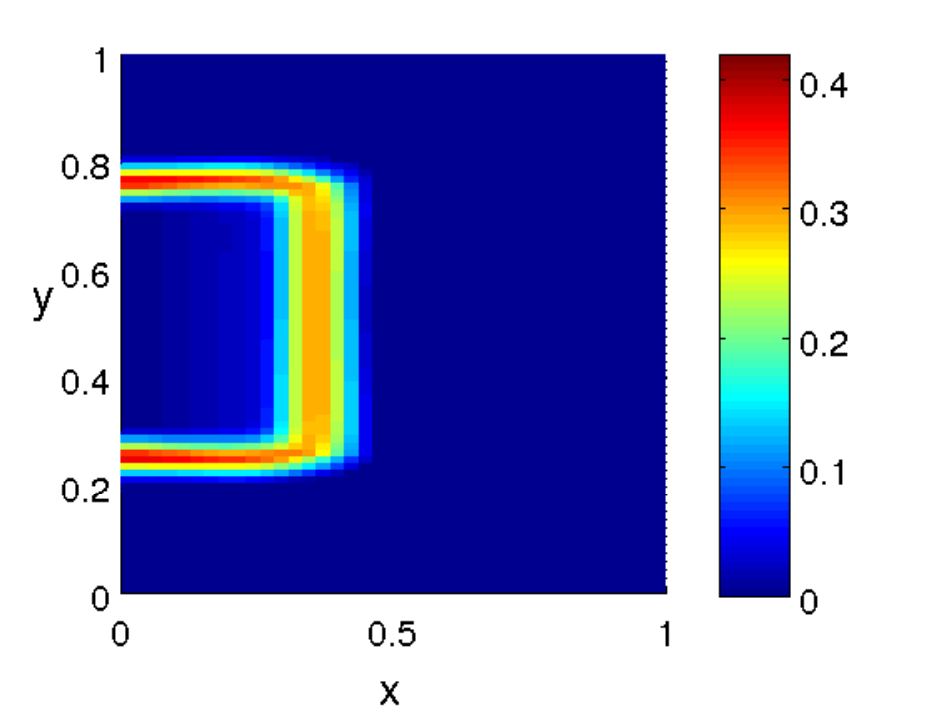}}
\subfigure[$t=0.5$, multiwavelets.]
{\includegraphics[width=0.48\textwidth]{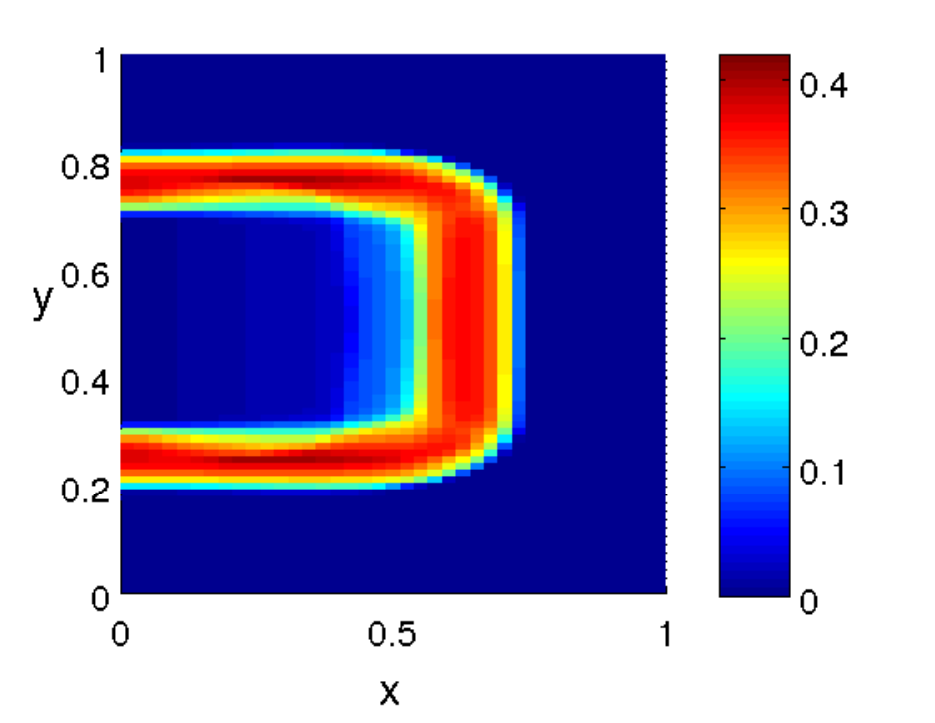}}
\subfigure[$t=0.5$, Monte Carlo.]
{\includegraphics[width=0.48\textwidth]{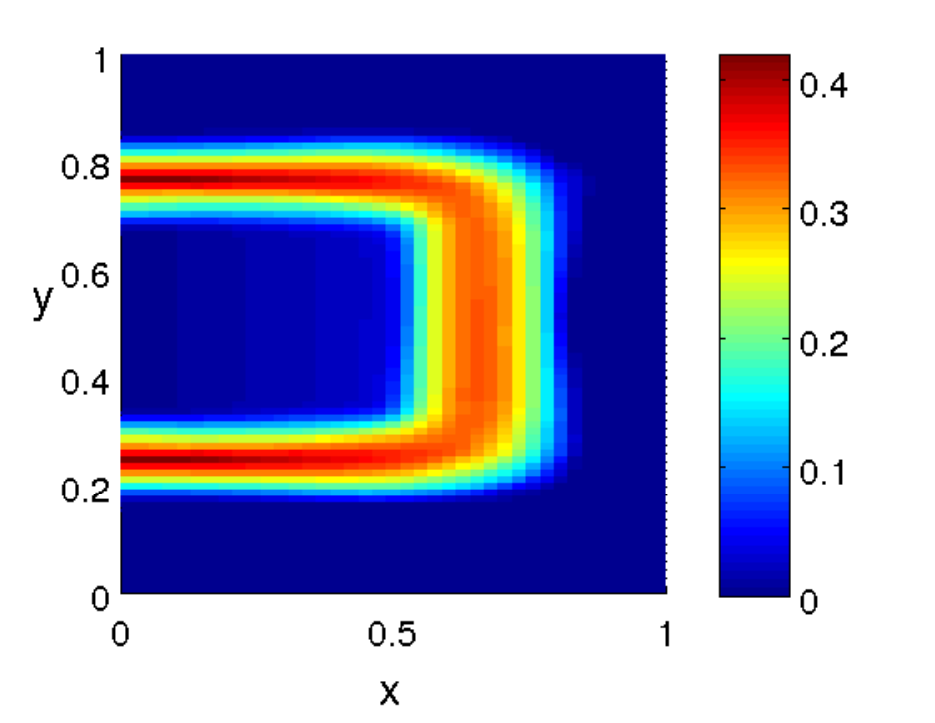}}
	\caption{Line injection. Standard deviation of the saturation, $t=0.25,0.5$.}
	\label{fig:line_std}
\end{figure}


Numerical results  are shown in Figure~\ref{fig:line_mean} for the mean value and in Figure~\ref{fig:line_std} for the standard deviation of the saturation. The velocity field is obtained from a permeability field of 4 random dimensions. A total-order basis of up to second-order Legendre polynomials was used for the multiwavelet saturation. The multiwavelet solution is compared to a Monte Carlo solution based on 1000 samples. This is a relatively small number, but adding more samples had a negligible effect on the Monte Carlo solution of this problem. The Monte Carlo samples are generated with a deterministic finite volume solver with the flux~\eqref{eq:hll_flux}. We use the minmod flux limiter which leads to the correct Kruzhkov entropy solution as observed in \citep{Kurganov_etal_07}. 

The uncertainties in both velocity components lead to uncertainty in the location of the saturation front. To resolve the variance in the $y$-direction, a relatively fine mesh was required compared to the $x$-direction ($m_y=80$, and $m_x=40$ grid cells, respectively). The MW solution requires finer resolution than any Monte Carlo sample solution due to small-scale variability in the higher order MW coefficients. The relative similarities of the Monte Carlo and MW solutions in Figures~\ref{fig:line_mean} and \ref{fig:line_std} indicate that the spatial mesh is sufficiently fine and that the retained order of MW is sufficient.


%
%

\subsection{Two Spatial Dimensions: Quarter Five-spot Problem}
\label{num_res_qfs}
We next consider a random velocity field where the mean field is defined by the quarter five-spot problem, i.e. streamlines originating from the lower left corner $(0,0)$ and ending at the upper right corner $(1,1)$ of the domain. No-flow is imposed over the domain boundaries. The setup is shown in Figure~\ref{fig:setup} (right). To obtain statistics for the saturation of the quarter five-spot problem, the multiwavelet representation of the total seepage velocity is determined in a preprocessing step. We represent the permeability field using the truncated KL expansion \eqref{eq:kl_exp} based on the assumption of a lognormal random field and an exponential covariance function. For all spatial grid cells, the permeability field is evaluated at stochastic quadrature points and used as input to an algebraic multi-grid pressure solver~\citep{Meyer_Tchelepi_10}. Assuming unit porosity, the MW coefficients of the total velocity or flux field are then approximated by numerical quadrature,
\[
(q_{(x)})_{k}(\bm{x}) = \int q_{(x)}\left(\bm{x}, \boldsymbol{\xi}\right) \psi_{k}(\boldsymbol{\xi})d\mathcal{P}(\boldsymbol{\xi}) \approx \sum_{j=1}^{N_q} q_{(x)}\left(\bm{x}, \boldsymbol{\xi}^{(j)}\right) \psi_{k}(\boldsymbol{\xi}^{(j)})w_j,
\]
where $\{ \boldsymbol{\xi}^{(j)} \}_{j=1}^{N_q}$ and $\{ w_{j} \}_{j=1}^{N_q}$ is the set of quadrature points and weights of a $N_q$-point quadrature rule. Note that the weights are accounting for the probability density function of the parameterization vector $\boldsymbol{\xi}$. We use the tensorized Gauss-Legendre quadrature rule in the Parameterized Matrix Package\footnote{Copyright 2009-2010, Paul G. Constantine and David F. Gleich.}.

Numerical results for the statistics of the saturation are shown in Figures~\ref{fig:q5s_mean} and \ref{fig:q5s_std}, for $t=0.5$, $t=1$, and $t=2$, respectively. The spatial mesh consists of 50 by 50 spatial cells and an order $P=15$ MW expansion is used for the saturation (total-order basis in two stochastic dimensions and up to fourth order polynomial reconstruction). The uncertain velocity field results in uncertain location of the front which is where the variance attains its maximum. An effect of the uncertainty in the saturation front is that the mean solution appears more smeared than any single realization of the saturation field. The slight smearing observed in the numerical results is thus a stochastic effect rather than a feature of the numerical method. The MW solution captures the mean and standard deviation. The difference from the reference Monte Carlo solution is most clearly visible in the standard deviation for large times. The numerical error is attributed to the truncation error of the MW expansion. In order to decrease the error, the order of MW expansion should be increased together with spatial mesh refinement since high-order MW modes vary on a shorter scale than low-order MW modes.

%
%

\begin{figure}[H]
\centering
\subfigure[$t=0.5$, multiwavelets.]
{\includegraphics[width=0.48\textwidth]{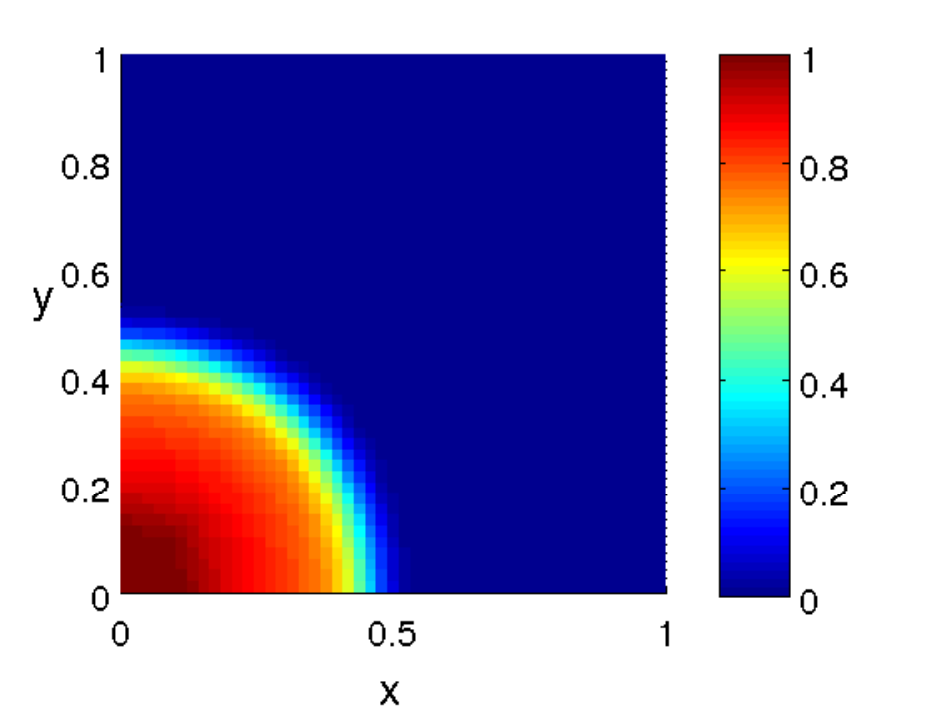}}
\subfigure[$t=0.5$, Monte Carlo.]
{\includegraphics[width=0.48\textwidth]{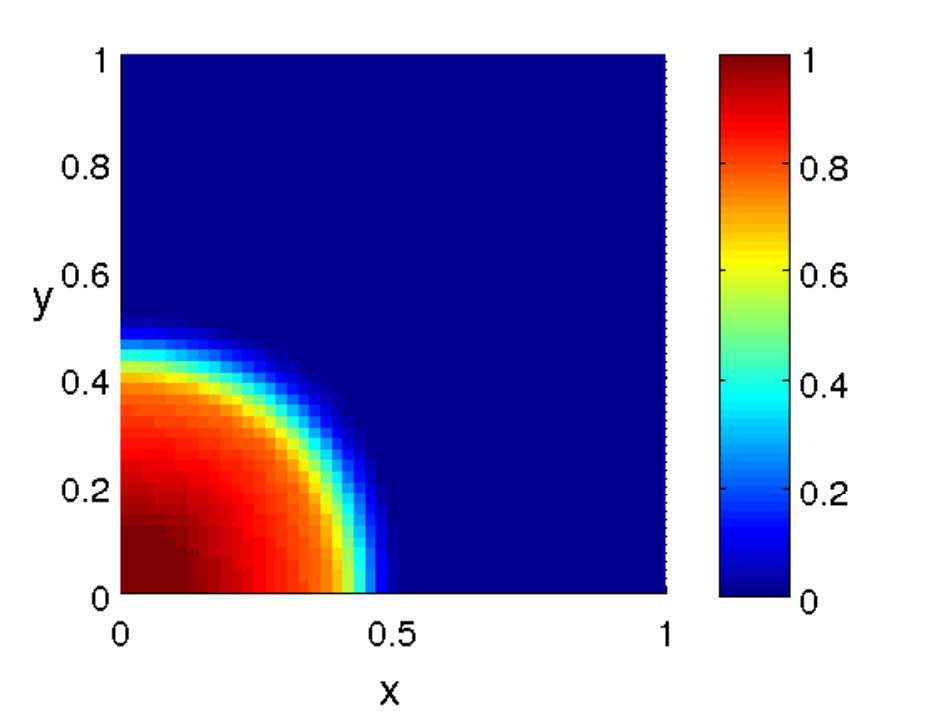}}
\subfigure[$t=1$, multiwavelets.]
{\includegraphics[width=0.48\textwidth]{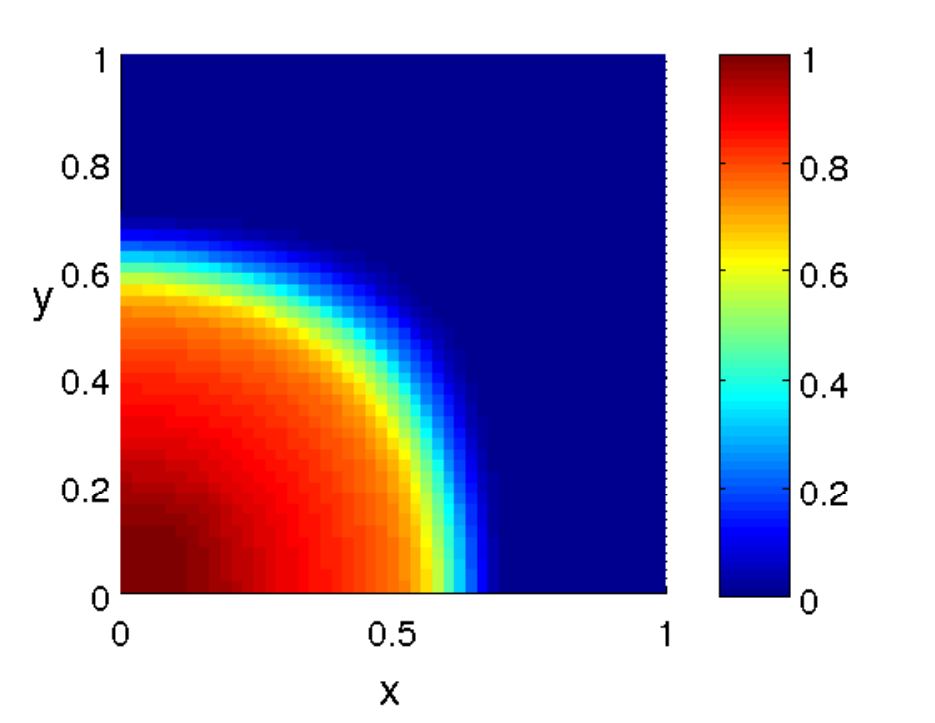}}
\subfigure[$t=1$, Monte Carlo.]
{\includegraphics[width=0.48\textwidth]{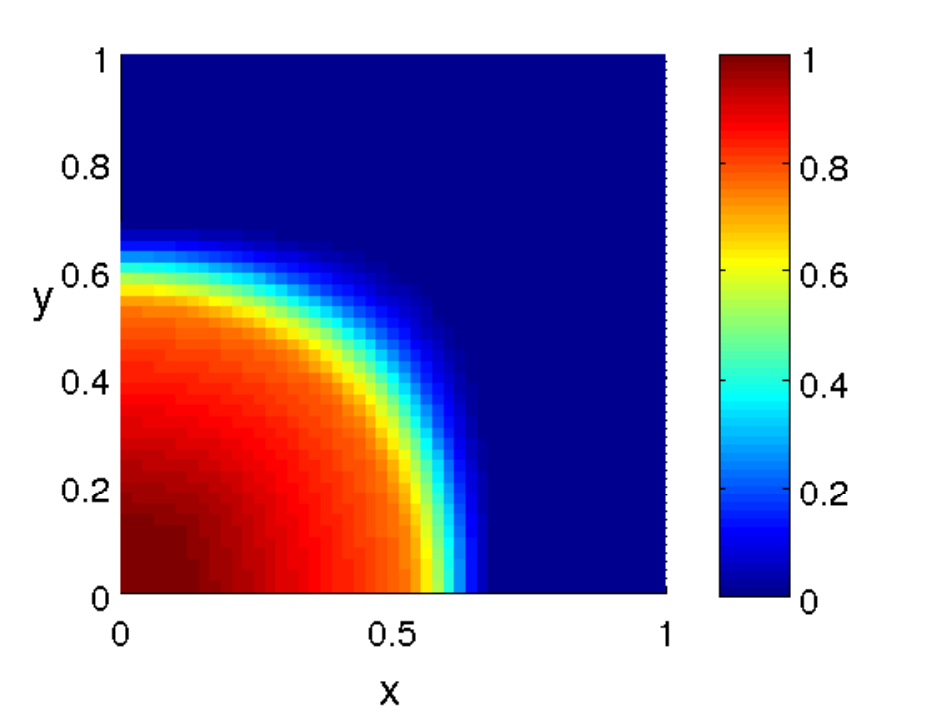}}
\subfigure[$t=2$, multiwavelets.]
{\includegraphics[width=0.48\textwidth]{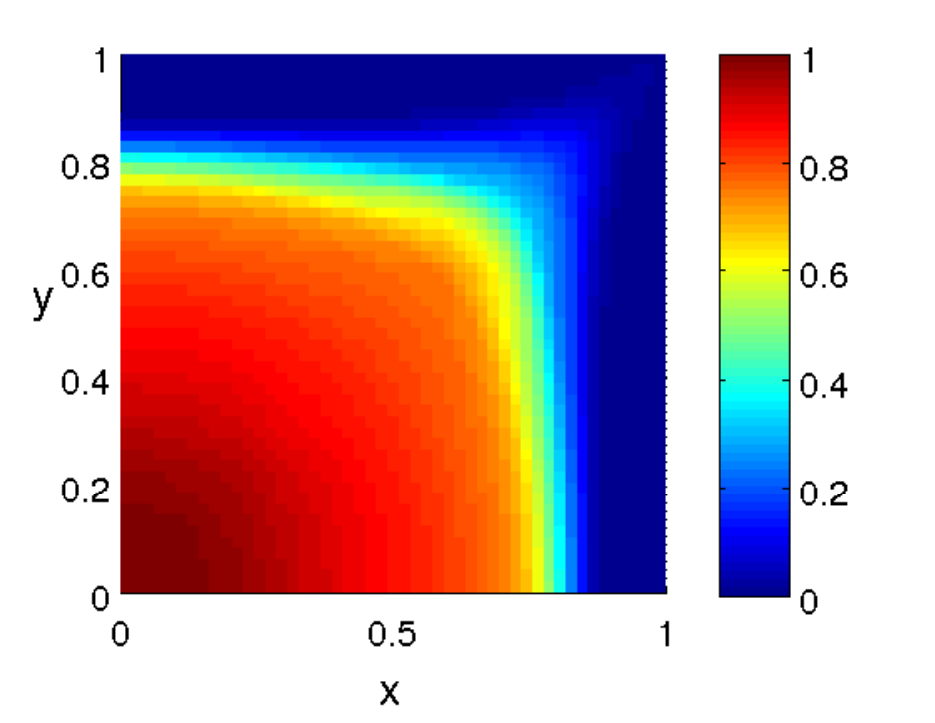}}
\subfigure[$t=2$, Monte Carlo.]
{\includegraphics[width=0.48\textwidth]{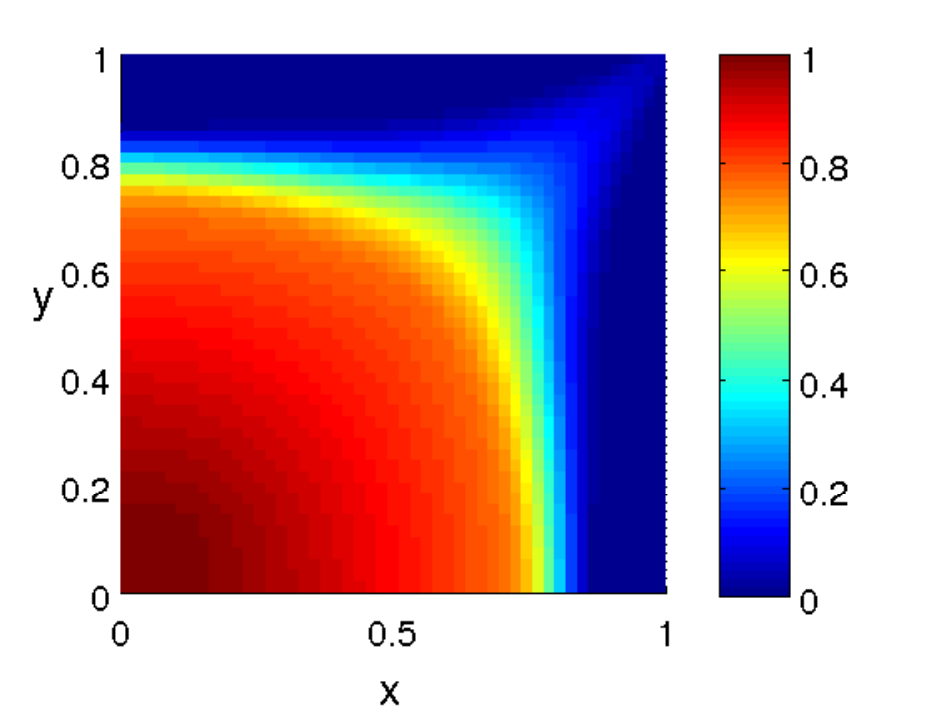}}
	\caption{Quarter-five-spot problem. Mean value of the saturation, $t=0.5,1,2$.}
	\label{fig:q5s_mean}
\end{figure}

\begin{figure}[H]
\centering
\subfigure[$t=0.5$, multiwavelets.]
{\includegraphics[width=0.48\textwidth]{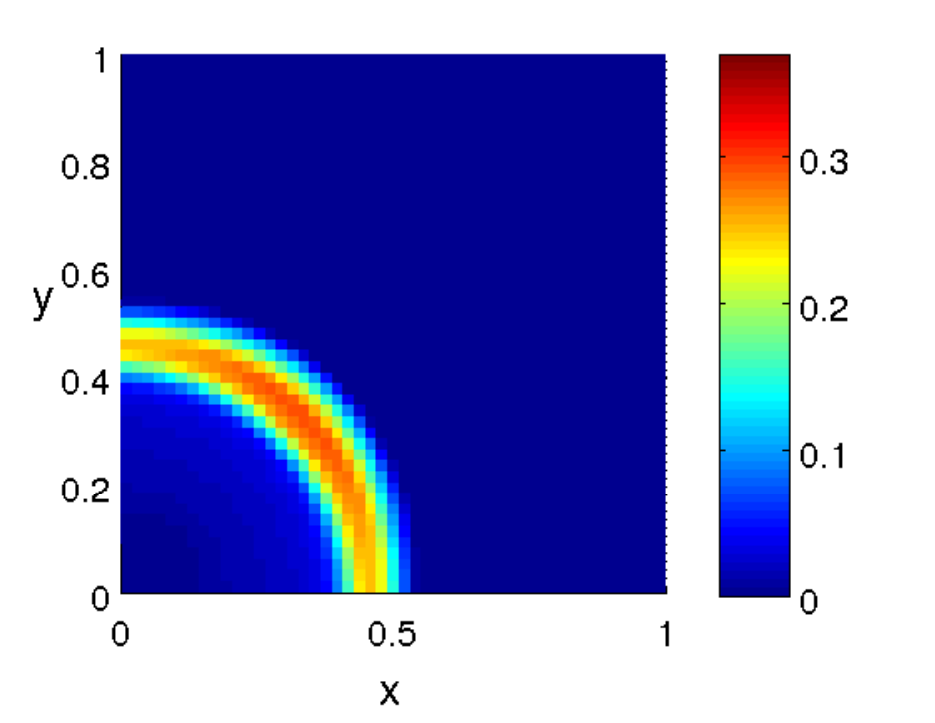}}
\subfigure[$t=0.5$, Monte Carlo.]
{\includegraphics[width=0.48\textwidth]{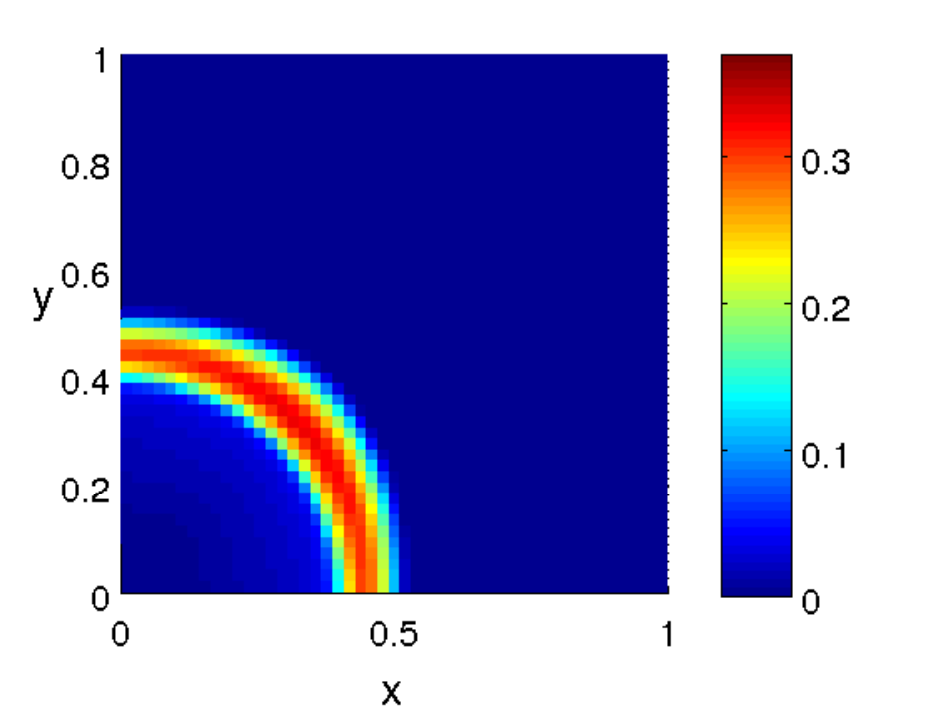}}
\subfigure[$t=1$, multiwavelets.]
{\includegraphics[width=0.48\textwidth]{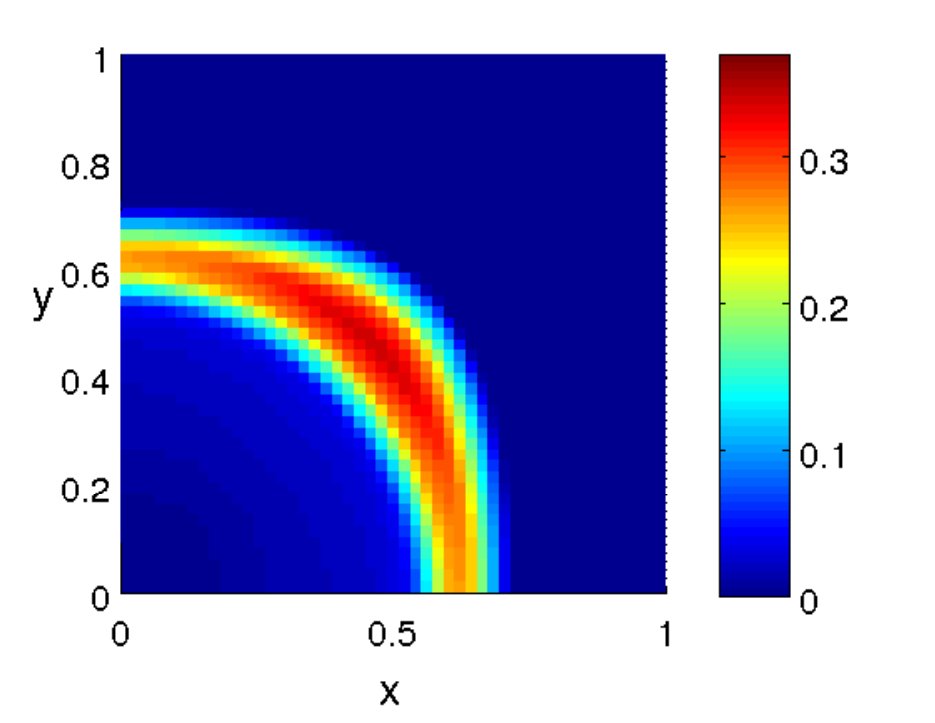}}
\subfigure[$t=1$, Monte Carlo.]
{\includegraphics[width=0.48\textwidth]{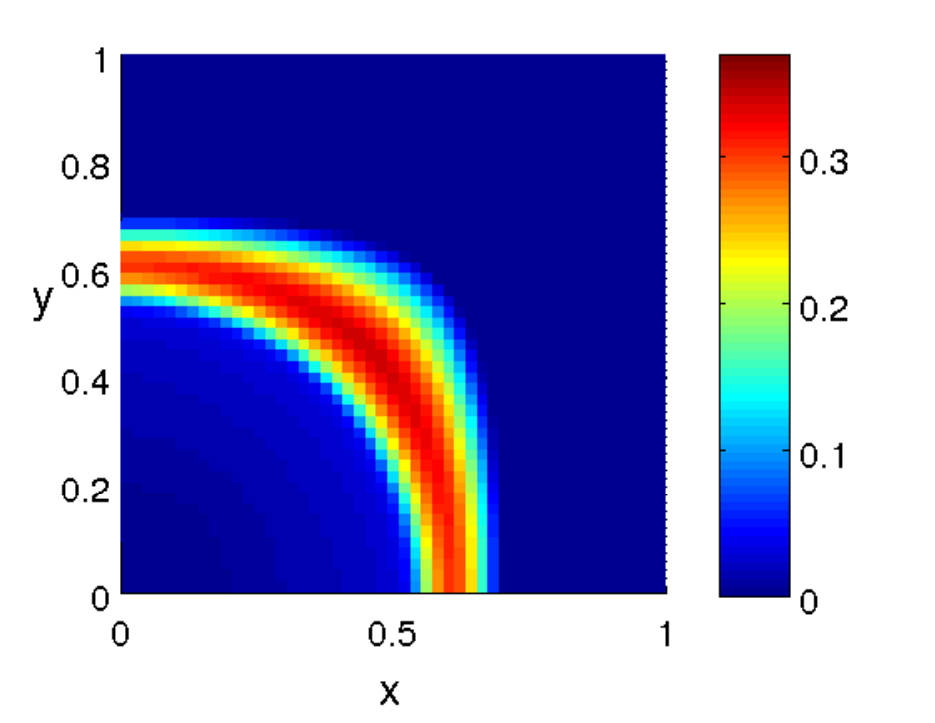}}
\subfigure[$t=2$, multiwavelets.]
{\includegraphics[width=0.48\textwidth]{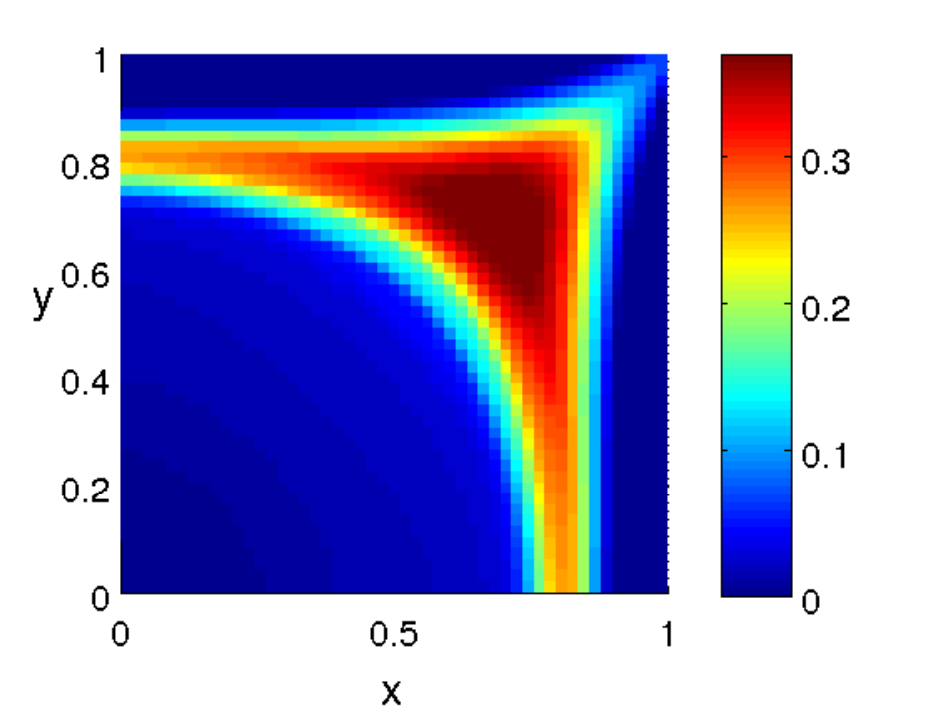}}
\subfigure[$t=2$, Monte Carlo.]
{\includegraphics[width=0.48\textwidth]{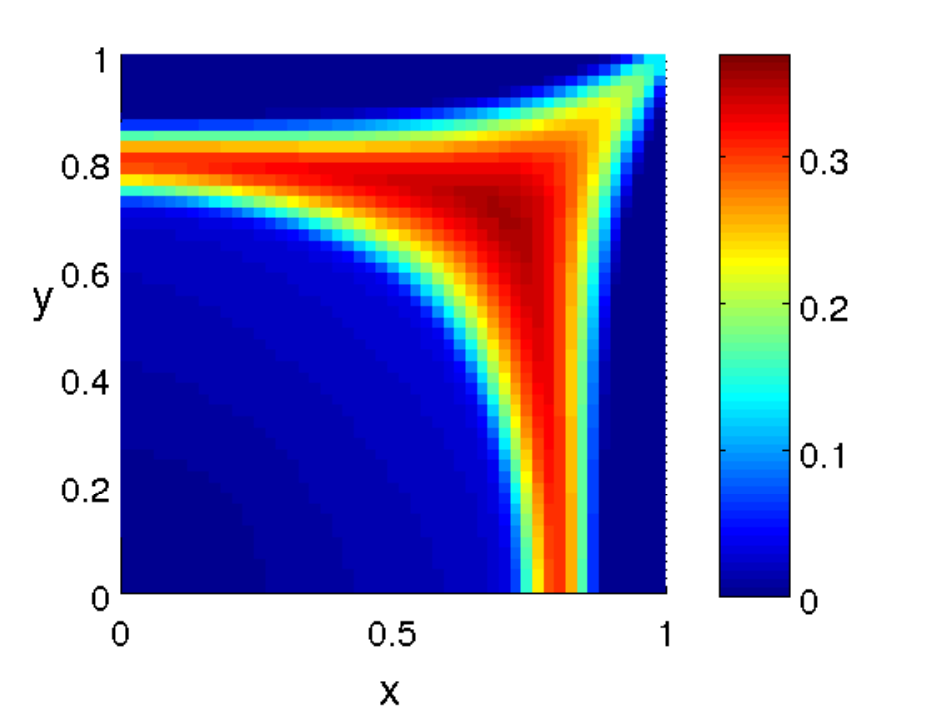}}
	\caption{Quarter-five-spot problem. Standard deviation of the saturation, $t=0.5,1,2$.}
	\label{fig:q5s_std}
\end{figure}


The locally reduced-order method is compared in terms of computational time to the standard stochastic Galerkin formulation by varying the order $P$ of MW expansion. The input velocity is obtained from a spatially varying permeability field with lognormal distribution and a KL expansion with 4 stochastic dimensions. The field is represented by multiwavelets that are not the optimal representation of this input uncertainty. An optimal representation (in the sense of minimum number of nonzero MW coefficients), i.e., basis functions tailored for this velocity field, would admit a more sparse representation and thus allow greater speedup of the locally reduced-order method. As can be inferred from the results displayed in Figure~\ref{fig:temp_comp_2D}, there is still a factor 6 speedup for large $P$.
\begin{figure}[H]
\centering
\subfigure 
{\includegraphics[width=0.48\textwidth]{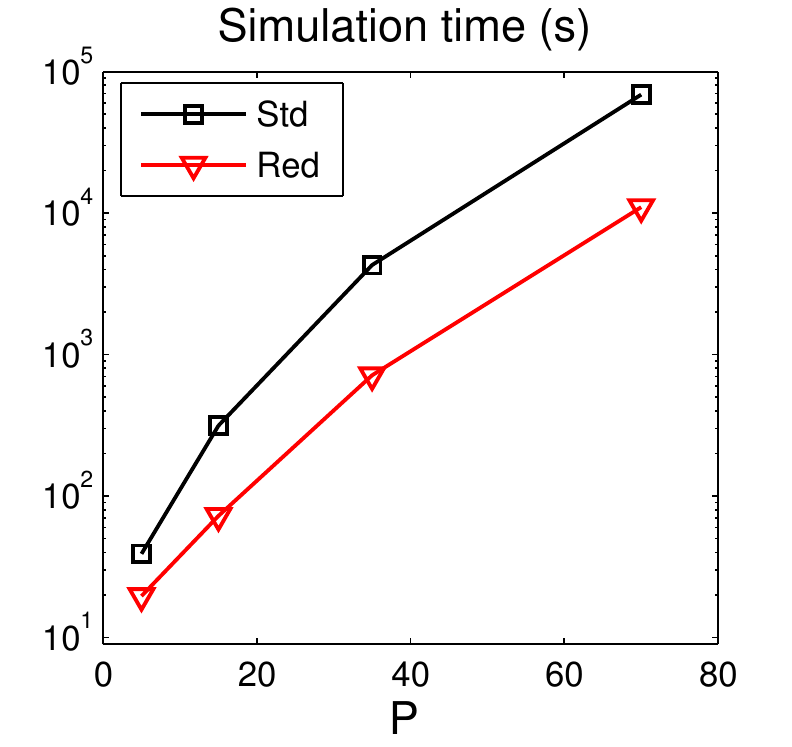}}
	\caption{Simulation times for the standard full-order (Std) and the reduced-order (Red) stochastic Galerkin formulations as a function of the order of MW expansion.}
	\label{fig:temp_comp_2D}
\end{figure}


\section{Conclusions}
We have shown that the fully intrusive stochastic Galerkin formulation of the Buckley-Leverett equation is hyperbolic independently of the order of the MW expansion. The approximate problem thus retains important qualitative features of the original stochastic formulation, e.g. wave propagation of the initial saturation and the emergence of a discontinuous solution in finite time. This analysis motivated the choice of a shock-capturing finite volume solver with a local reduction of the order of MW expansion for computational efficiency. For the orders of MW investigated, the numerical solution converges to the reference solution generated by Monte Carlo sampling.

The numerical solver performs well for the stochastic Galerkin formulation of the Buckley-Leverett problem with the pseudo-spectral flux approximation and multiwavelet representation of the uncertainty. For sufficient spatial resolution, the numerical solutions display multiple discontinuities depending on the order $P$ of stochastic truncation. In large-scale problems with multiple sources of uncertainty, discontinuities may not be visible due to lack of resolution and smoothing effects from stochastic averaging. Nevertheless, only a robust numerical scheme with shock-capturing properties can produce a solution that converges with mesh refinement.

Two spatially two-dimensional problems have been investigated: a line-injection problem and a quarter five-spot problem. For the line injection problem, we have shown how to incorporate analytically derived correlation functions in the MW framework. For the 
quarter five-spot problem, the MW representation of the input velocity field was obtained from a lognormal permeability field with an algebraic multigrid solver. The multiple discontinuities are not resolved due to the high order of MW in either problem, but the resolution is still sufficient for good agreement with the Monte Carlo reference solutions.

We have demonstrated substantial speedup for a locally reduced-order basis stochastic Galerkin method, as compared to the standard implementation of the stochastic Galerkin method with global basis functions. The rationale for the method presented here is that solutions of hyperbolic systems consist of waves and thus are localized in space and time. At a given point in space and time, only a subset of all waves contribute significantly to the solution. The method presented is adaptive in the sense that it identifies locally significant waves and use only them in the numerical flux computations. An attractive feature of the method is that no apriori information about the solution is assumed. All stochastic modes are updated and this ensures that temporarily insignificant modes are accounted for as soon as they grow above a given threshold. The method is  expected to be efficient also for other problems described by wave propagation, e.g., computational fluid dynamics. This will be investigated in more detail elsewhere.

Future work also includes evaluation of the presented local basis reduction method on stochastic permeability fields conditioned on measurements. The adaptivity in the complexity of the stochastic representation should have a significant impact on the total computational cost since the optimal stochastic representation is expected to be highly localized in a spatially heterogeneous field.

\section*{Acknowledgements}
The authors would like to gratefully acknowledge the Reservoir Simulation Research Program (SUPRI-B) at Stanford University for funding of this project, and the first author would also like to acknowledge the SUCCESS centre for CO$_2$ storage under grant 193825/S60 from the Research Council of Norway (RCN) for funding. We also would like to thank Dr. Daniel Meyer and Dr. Anna Nissen for help with the design of the quarter five-spot problem and Fayadhoi Ibrahima for valuable suggestions and advice on the presentation of the material.



\bibliographystyle{plain}

\begin{thebibliography}{10}
\expandafter\ifx\csname url\endcsname\relax
  \def\url#1{\texttt{#1}}\fi
\expandafter\ifx\csname urlprefix\endcsname\relax\def\urlprefix{URL }\fi
\expandafter\ifx\csname href\endcsname\relax
  \def\href#1#2{#2} \def\path#1{#1}\fi

\bibitem{Gelhar_86}
L.~W. Gelhar, Stochastic subsurface hydrology from theory to applications,
  Water Resour. Res. 22~(9S) (1986) 135S--145S.
\newblock \href {http://dx.doi.org/10.1029/WR022i09Sp0135S}
  {\path{doi:10.1029/WR022i09Sp0135S}}.

\bibitem{Zhang_02}
D.~Zhang, Stochastic Methods for Flow in Porous Media: Coping with
  Uncertainties, Academic Press, San Diego, Calif., 2002.

\bibitem{Neuman_Orr_93}
S.~P. Neuman, S.~Orr, Prediction of steady state flow in nonuniform geologic
  media by conditional moments: Exact nonlocal formalism, effective
  conductivities, and weak approximation, Water Resour. Res. 29~(2) (1993)
  341--364.
\newblock \href {http://dx.doi.org/10.1029/92WR02062}
  {\path{doi:10.1029/92WR02062}}.

\bibitem{Guadagnini_Neuman_99}
A.~Guadagnini, S.~P. Neuman, Nonlocal and localized analyses of conditional
  mean steady state flow in bounded, randomly nonuniform domains: 1. {T}heory
  and computational approach, Water Resour. Res. 35~(10) (1999) 2999--3018.
\newblock \href {http://dx.doi.org/10.1029/1999WR900160}
  {\path{doi:10.1029/1999WR900160}}.

\bibitem{Zhang_Lu_04}
D.~Zhang, Z.~Lu, An efficient, high-order perturbation approach for flow in
  random porous media via {K}arhunen-{L}o\`{e}ve and polynomial expansions, J.
  Comput. Phys. 194~(2) (2004) 773 -- 794.

\bibitem{Jarman_Tartakovsky_13}
K.~D. Jarman, A.~M. Tartakovsky, A comparison of closures for stochastic
  advection-diffusion equations, SIAM/ASA Journal on Uncertainty Quantification
  1~(1) (2013) 319--347.
\newblock \href {http://dx.doi.org/10.1137/120897419}
  {\path{doi:10.1137/120897419}}.

\bibitem{Caroni_Fiorotto_05}
E.~Caroni, V.~Fiorotto, Analysis of concentration as sampled in natural
  aquifers, Transport Porous Med. 59~(1) (2005) 19--45.
\newblock \href {http://dx.doi.org/10.1007/s11242-004-1119-x}
  {\path{doi:10.1007/s11242-004-1119-x}}.

\bibitem{Zhang_Tchelepi_99}
D.~Zhang, H.~A. Tchelepi, Stochastic analysis of immiscible two-phase flow in
  heterogeneous media, SPE Journal 4~(4) (1999) 380--388.

\bibitem{Zhang_Li_Tchelepi_00}
D.~Zhang, L.~Li, H.~A. Tchelepi, Stochastic formulation for uncertainty
  analysis of two-phase flow in heterogeneous reservoirs, SPE Journal 5~(1)
  (2000) 60--70.

\bibitem{Wang_etal_13}
P.~Wang, D.~M. Tartakovsky, K.~D. {Jarman, Jr.}, A.~M. Tartakovsky, {CDF}
  solutions of {Buckley--Leverett} equation with uncertain parameters,
  Multiscale Model. and Simul. 11~(1) (2013) 118--133.

\bibitem{Ghanem:1991}
R.~G. Ghanem, P.~D. Spanos, Stochastic finite elements: a spectral approach,
  Springer-Verlag, New York, 1991.

\bibitem{XiuKarniadakis02}
D.~Xiu, G.~E. Karniadakis, The {W}iener--{A}skey polynomial chaos for
  stochastic differential equations, SIAM J. Sci. Comput. 24~(2) (2002)
  619--644.
\newblock \href {http://dx.doi.org/http://dx.doi.org/10.1137/S1064827501387826}
  {\path{doi:http://dx.doi.org/10.1137/S1064827501387826}}.

\bibitem{Oladyshkin_etal_2011}
S.~Oladyshkin, H.~Class, R.~Helmig, W.~Nowak, A concept for data-driven
  uncertainty quantification and its application to carbon dioxide storage in
  geological formations, Adv. Water Resour. 34~(11) (2011) 1508--1518.

\bibitem{Sochala_LeMaitre_13}
P.~Sochala, O.~P. Le~Ma\^{i}tre, Polynomial chaos expansion for subsurface
  flows with uncertain soil parameters, Adv. Water Resour. 62, Part A (2013)
  139 -- 154.
\newblock \href
  {http://dx.doi.org/http://dx.doi.org/10.1016/j.advwatres.2013.10.003}
  {\path{doi:http://dx.doi.org/10.1016/j.advwatres.2013.10.003}}.

\bibitem{Li_Zhang_07}
H.~Li, D.~Zhang, Probabilistic collocation method for flow in porous media:
  Comparisons with other stochastic methods, Water Resour. Res. 43 (2007)
  1--13.

\bibitem{Muller_etal_11}
F.~M\"{u}ller, P.~Jenny, D.~W. Meyer, Probabilistic collocation and
  {L}agrangian sampling for advective tracer transport in randomly
  heterogeneous porous media, Adv. Water Resour. 34~(12) (2011) 1527 -- 1538.
\newblock \href {http://dx.doi.org/10.1016/j.advwatres.2011.09.005}
  {\path{doi:10.1016/j.advwatres.2011.09.005}}.

\bibitem{LeMaitre_etal_04}
O.~P. Le~Ma\^{\i}tre, H.~N. Najm, R.~G. Ghanem, O.~M. Knio, Multi-resolution
  analysis of {W}iener-type uncertainty propagation schemes, J. Comput. Phys.
  197 (2004) 502--531.
\newblock \href {http://dx.doi.org/10.1016/j.jcp.2003.12.020}
  {\path{doi:10.1016/j.jcp.2003.12.020}}.

\bibitem{Wan_Karniadakis_05}
X.~Wan, G.~E. Karniadakis, An adaptive multi-element generalized polynomial
  chaos method for stochastic differential equations, J. Comput. Phys. 209
  (2005) 617--642.
\newblock \href {http://dx.doi.org/http://dx.doi.org/10.1016/j.jcp.2005.03.023}
  {\path{doi:http://dx.doi.org/10.1016/j.jcp.2005.03.023}}.

\bibitem{Koppel_etal_14}
M.~K\"{o}ppel, I.~Kr\"{o}ker, C.~Rohde, Stochastic modeling for heterogeneous
  two-phase flow, in: J.~Fuhrmann, M.~Ohlberger, C.~Rohde (Eds.), Finite
  Volumes for Complex Applications VII-Methods and Theoretical Aspects,
  Vol.~77, Springer International Publishing, 2014, pp. 353--361.

\bibitem{Pettersson_Tchelepi_14}
P.~Pettersson, H.~A. Tchelepi, Stochastic {G}alerkin method for the
  {B}uckley-{L}everett problem in heterogeneous formations, in: Proc. 14th
  European Conference on the Mathematics of Oil Recovery ECMOR XIV, no. A33,
  EAGE, Catania, Italy, 2014.

\bibitem{Xiu_Hesthaven_05}
D.~Xiu, J.~S. Hesthaven, High-order collocation methods for differential
  equations with random inputs, SIAM J. Sci. Comput. 27 (2005) 1118--1139.
\newblock \href {http://dx.doi.org/10.1137/040615201}
  {\path{doi:10.1137/040615201}}.

\bibitem{Burger_etal_14}
R.~B\"{u}rger, I.~Kr\"{o}ker, C.~Rohde, A hybrid stochastic {G}alerkin method
  for uncertainty quantification applied to a conservation law modelling a
  clarifier-thickener unit, J. Appl. Math. Mech.-ZAMM 94~(10) (2014) 793--817.
\newblock \href {http://dx.doi.org/10.1002/zamm.201200174}
  {\path{doi:10.1002/zamm.201200174}}.

\bibitem{Tryoen_etal_10}
J.~Tryoen, O.~P. Le~Ma\^{\i}tre, M.~Ndjinga, A.~Ern, Intrusive {G}alerkin
  methods with upwinding for uncertain nonlinear hyperbolic systems, J. Comput.
  Phys. 229~(18) (2010) 6485 -- 6511.
\newblock \href {http://dx.doi.org/DOI: 10.1016/j.jcp.2010.05.007}
  {\path{doi:DOI: 10.1016/j.jcp.2010.05.007}}.

\bibitem{Rubin_wrr_90}
Y.~Rubin, Stochastic modeling of macrodispersion in heterogeneous porous media,
  Water Resour. Res. 26 (1990) 133--141.

\bibitem{ka46}
K.~Karhunen, {Z}ur {S}pektraltheorie stochastischer {P}rozesse., {A}nn. {A}cad.
  {S}ci. {F}enn. {S}er. {A} {I} {M}ath. 34 (1946) 3--7.

\bibitem{Loeve_77}
M.~Lo\`{e}ve, Probability Theory, 4th Edition, Springer-Verlag, 1977.

\bibitem{Hoeksema_Kitanidis_85}
R.~J. Hoeksema, P.~K. Kitanidis, Analysis of the spatial structure of
  properties of selected aquifers, Water Resour. Res. 21~(4) (1985) 563--572.
\newblock \href {http://dx.doi.org/10.1029/WR021i004p00563}
  {\path{doi:10.1029/WR021i004p00563}}.

\bibitem{Perrin_etal_13}
G.~Perrin, C.~Soize, D.~Duhamel, C.~Funfschilling, {K}arhunen-{L}o\`{e}ve
  expansion revisited for vector-valued random fields: Scaling, errors and
  optimal basis, J. Comput. Phys. 242 (2013) 607 -- 622.

\bibitem{Smolyak_63}
S.~Smolyak, Quadrature and interpolation formulas for tensor products of
  certain classes of functions, Soviet Mathematics, Doklady 4 (1963) 240--243.

\bibitem{Alpert_93}
B.~K. Alpert, A class of bases in {L}2 for the sparse representations of
  integral operators, SIAM J. Math. Anal. 24~(1) (1993) 246--262.
\newblock \href {http://dx.doi.org/10.1137/0524016}
  {\path{doi:10.1137/0524016}}.

\bibitem{Debusschere}
B.~J. Debusschere, H.~N. Najm, P.~P. P\'{e}bay, O.~M. Knio, R.~G. Ghanem, O.~P.
  Le~Ma\^{\i}tre, Numerical challenges in the use of polynomial chaos
  representations for stochastic processes, SIAM J. Sci. Comput. 26 (2005)
  698--719.
\newblock \href {http://dx.doi.org/http://dx.doi.org/10.1137/S1064827503427741}
  {\path{doi:http://dx.doi.org/10.1137/S1064827503427741}}.

\bibitem{Harten_etal_83}
A.~Harten, P.~D. Lax, B.~van Leer, On upstream differencing and {G}odunov-type
  schemes for hyperbolic conservation laws, SIAM Review 25~(1) (1983) 35--61.

\bibitem{Einfeld:1988}
B.~Einfeld, On {G}odunov-type methods for gas dynamics, SIAM J. Numer. Anal.
  25~(2) (1988) 294--318.
\newblock \href {http://dx.doi.org/10.1137/0725021}
  {\path{doi:10.1137/0725021}}.

\bibitem{Kurganov_etal_01}
A.~Kurganov, S.~Noelle, G.~Petrova, Semidiscrete central-upwind schemes for
  hyperbolic conservation laws and {H}amilton--{J}acobi equations, SIAM J. Sci.
  Comput. 23~(3) (2001) 707--740.
\newblock \href {http://dx.doi.org/10.1137/S1064827500373413}
  {\path{doi:10.1137/S1064827500373413}}.

\bibitem{Kurganov_etal_07}
A.~Kurganov, G.~Petrova, B.~Popov, Adaptive semidiscrete central-upwind schemes
  for nonconvex hyperbolic conservation laws, SIAM J. Sci. Comput. 29~(6)
  (2007) 2381--2401.
\newblock \href {http://dx.doi.org/10.1137/040614189}
  {\path{doi:10.1137/040614189}}.

\bibitem{Suresh_00}
A.~Suresh, Positivity-preserving schemes in multidimensions, SIAM J. Sci.
  Comput. 22~(4) (2000) 1184--1198.
\newblock \href {http://dx.doi.org/10.1137/S1064827599360443}
  {\path{doi:10.1137/S1064827599360443}}.

\bibitem{Wan_Karniadakis_05_b}
X.~Wan, G.~E. Karniadakis, Long-term behavior of polynomial chaos in stochastic
  flow simulations, Comput. Methods Appl. Math. Eng. 195 (2006) 5582 -- 5596.
\newblock \href {http://dx.doi.org/http://dx.doi.org/10.1016/j.cma.2005.10.016}
  {\path{doi:http://dx.doi.org/10.1016/j.cma.2005.10.016}}.

\bibitem{Schwab_Todor_06}
C.~Schwab, R.~A. Todor, Karhunen-{L}o\`{e}ve approximation of random fields by
  generalized fast multipole methods, J. Comput. Phys. 217~(1) (2006) 100--122.

\bibitem{Meyer_Tchelepi_10}
D.~W. Meyer, H.~A. Tchelepi, Particle-based transport model with {M}arkovian
  velocity processes for tracer dispersion in highly heterogeneous porous
  media, Water Resour. Res. 46~(11), w11552.
\newblock \href {http://dx.doi.org/10.1029/2009WR008925}
  {\path{doi:10.1029/2009WR008925}}.

\end{thebibliography}

\end{document}